\documentclass[11pt,reqno]{amsart}
\usepackage{comment}
\usepackage{amssymb}
\usepackage{tikz-cd}
\usepackage{tikz}
\usepackage{amsthm}
\usepackage{graphicx}
\usepackage{caption}
\usepackage{subcaption}
\usepackage[backref]{hyperref}
\usepackage{cleveref} % Required for inserting images
\usepackage{booktabs}
\usepackage{float}

\usepackage[disable]{todonotes}
\DeclareMathOperator{\rk}{rk}
\DeclareMathOperator{\Aut}{Aut}

\DeclareMathOperator{\End}{End}
\DeclareMathOperator{\Gal}{Gal}

\DeclareMathOperator{\PSL}{PSL}
\DeclareMathOperator{\GL}{GL}

\DeclareMathOperator{\Id}{Id}
\DeclareMathOperator{\SL}{SL}

\DeclareMathOperator{\Hom}{Hom}

\DeclareMathOperator{\Frob}{Frob}

\DeclareMathOperator{\Ver}{Ver}

\DeclareMathOperator{\ddiv}{div}

\DeclareMathOperator{\red}{red}
\DeclareMathOperator{\gon}{gon}
\DeclareMathOperator{\img}{im}
\DeclareMathOperator{\Spec}{Spec}
\DeclareMathOperator{\Ell}{Ell}

\newtheorem{theorem}{Theorem}[section]
\newtheorem*{theorem*}{Theorem}
\newtheorem{lemma}[theorem]{Lemma}
\newtheorem{conjecture}[theorem]{Conjecture}
\newtheorem{proposition}[theorem]{Proposition}
\newtheorem*{proposition*}{Proposition}
\newtheorem{corollary}[theorem]{Corollary}
\newtheorem*{corollary*}{Corollary}

\theoremstyle{definition}
\newtheorem{remark}[theorem]{Remark}

\newtheorem{example}[theorem]{Example}

\numberwithin{equation}{section}

\newcommand{\Q}{\mathbb{Q}}
\newcommand{\Z}{\mathbb{Z}}

\newcommand{\C}{\mathbb{C}}
\newcommand{\F}{\mathbb{F}}
\newcommand{\A}{\mathbb{A}}
\newcommand{\PP}{\mathbb{P}}

\newcommand{\Qbar}{\overline \Q}

\newcommand{\Hamil}{\mathbb{H}}

\newcommand{\fp}{\mathfrak{p}}

\newcommand{\OO}{\mathcal{O}}

\newcommand{\tor}{\mathrm{tors}}

\def\torz#1{\mathbb Z /#1 \mathbb Z}
\def\tg#1#2{\mathbb Z/#1\mathbb Z \times \mathbb Z /#2 \mathbb Z}
\newcommand{\diamondop}[1]{\langle #1 \rangle} % diamond operator
\newcommand{\set}[1]{\left\lbrace #1 \right\rbrace}

\newcommand{\lmfdcharorbit}[2]{\href{https://www.lmfdb.org/Character/Dirichlet/#1/#2}{#1.#2}}
\newcommand{\lmfdbnewform}[4]{\href{https://www.lmfdb.org/ModularForm/GL2/Q/holomorphic/#1/#2/#3/#4}{#1.#2.#3.#4}}
\newcommand{\githubbare}[1]{\href{https://github.com/nt-lib/quartic-torsion/blob/main/#1}{\path{#1}}}
\newcommand{\github}[2]{[\githubbare{#1}, #2]}

\title{Classification of torsion of elliptic curves over quartic fields}
\date{\today}

\begin{document}

\begin{abstract}
Let $E$ be an elliptic curve over a quartic field $K$. By the Mordell-Weil theorem, $E(K)$ is a finitely generated group. We determine all the possibilities for the torsion group $E(K)_\tor$ where $K$ ranges over all quartic fields $K$ and $E$ ranges over all elliptic curves over $K$. We show that there are no sporadic torsion groups, or in other words, that all torsion groups either do not appear or they appear for infinitely many non-isomorphic elliptic curves $E$.

Proving this requires showing that numerous modular curves $X_1(m,n)$ have no non-cuspidal degree $4$ points. We deal with almost all the curves using one of 3 methods: a method for the rank 0 cases requiring no computation (\Cref{prop_main}); the Hecke sieve, a local method requiring computer-assisted computations; and the global method, an argument for the positive rank cases also requiring no computation. We deal with the handful of remaining cases using ad hoc methods.

\end{abstract}

\author[Derickx]{Maarten Derickx}
\address{Maarten Derickx, University of Zagreb, Faculty of Science, Department of Mathematics, Bijeni\v{c}ka Cesta 30, 10000 Zagreb, Croatia}
\email{\url{maarten@mderickx.nl}}
\urladdr{\url{https://www.maartenderickx.nl/}}

\author[Najman]{Filip Najman}
\address{Filip Najman, University of Zagreb, Faculty of Science, Department of Mathematics, Bijeni\v{c}ka Cesta 30, 10000 Zagreb, Croatia}
\email{\url{fnajman@math.hr}}
\urladdr{\url{https://web.math.pmf.unizg.hr/~fnajman/}}
\thanks{The authors were supported by the  project “Implementation of cutting-edge research and its application as part of the Scientific Center of Excellence for Quantum and Complex Systems, and Representations of Lie Algebras“, PK.1.1.02, European Union, European Regional Development Fund and by the Croatian Science Foundation under the project no. IP-2022-10-5008.}

\maketitle

\section{Introduction}
Let $E$ be an elliptic curve over a number field $K$. The set $E(K)$ of $K$-rational points on $E$ has, by the Mordell-Weil theorem, the structure of a finitely generated abelian group. Hence it is isomorphic to $E(K)_\tor \times \Z^r$, where $E(K)_\tor$ is the \textit{torsion subgroup} of $E(K)$, and $r$ is the \textit{rank} of $E$ over $K$. The question of what are the possible values of $E(K)_\tor$ for $E$ running through all elliptic curves over all number fields $K$ of degree $d$ has received considerable attention. The first major result was Mazur's torsion theorem \cite{mazur77}, which, building on work of Kubert \cite{Kubert76} (and others), completed the classification of possible torsion groups of elliptic curves over $\Q.$ For quadratic fields, the classification was completed by Kamienny \cite{kamienny92}, building on work of Kenku and Momose \cite{KM88}. The last degree $d$ for which all the possible torsion groups of elliptic curves over degree $d$ number fields are known is $d=3$. The result was proved by Derickx, Etropolski, van Hoeij, Morrow and Zureick-Brown in \cite{Deg3Class}, building on previous work of Jeon, Kim and Schweizer \cite{JeonKimSchweizer04}, and Bruin and Najman \cite{BruinNajman15}.

The approach in all the mentioned papers is to study degree $d$ points on the modular curves $X_1(m,n)$. When $m > 4$ or $n > 4$ these are fine moduli spaces of (generalized) elliptic curves of triples $(E,P,Q)$ with $P$ and $Q$ being points on $E$ generating a subgroup isomorphic to $\tg{m}{n}$.

In this paper, we complete the classification of possible torsion groups of elliptic curves over quartic fields. Our main result is the following theorem.
\begin{theorem}\label{main_theorem}
 If $K$ varies over all quartic number fields and $E$ varies over
all elliptic curves over $K$, the groups that appear as
$E(K)_{\tor}$ are exactly the following
\begin{alignat*}{2}
& \torz{n},  && n=1-18, 20,21,22,24, \\
& \tg{2}{2n}, \quad \quad && n=1-9, \\
& \tg{3}{3n}, \quad && n=1-3, \\
& \tg{4}{4n}, \quad && n=1,2, \\
& \tg{5}{5}, \quad && \\
& \tg{6}{6}. \quad && \\
\end{alignat*}
\end{theorem}
One additional reason that the classification of possible torsion groups over quartic fields might be particularly interesting is that it seems it might be the last degree where there are no non-cuspidal \textit{sporadic points} on $X_1(m,n)$. Recall that a sporadic point on a curve $X$ is a point of degree $d$ such that there are only finitely many points of degree $\leq d$ on $X$. There are no non-cuspidal sporadic points on $X_1(m,n)$ of degree $1$ and $2$, while there exist sporadic cubic points on only one such curve, $X_1(21)$, and all of them correspond to a single elliptic curve (see \cite{najman16}). Jeon, Kim and Park \cite{JeonKimPark06} proved that all the groups in \Cref{main_theorem} appear infinitely often (see \Cref{infinite_quartic}), so it follows that there are no sporadic points of degree $4$ on $X_1(m,n)$. Moreover, from van Hoeij's results \cite{vanHoeij} one can easily see that there exist sporadic points of degree $5\leq d\leq 13$. See \Cref{sec:sporadic} for more details. Similarly, Clark, Genao, Pollack and Saia \cite{CGPS22} showed that $X_1(n)$ has sporadic points (corresponding to CM elliptic curves) for all $n>720$. These results show that there is an abundance of sporadic points on $X_1(n)$, and it might well be possible that $d=4$ is the largest degree such that, for all $n$, $X_1(n)$ has no degree $d$ sporadic points.

We use several tools to prove \Cref{main_theorem}. To deal with the cases where $\rk J_1(m,n)(\Q)=0,$ we first prove \Cref{prop_main} which allows us to eliminate all the cases where the hypothetical $E/K$ with $\tg{m}{n}\subseteq E(K)$ has bad reduction at all places above some rational prime $p$, e.g. when $Y_1(m,n)(\F_{p^d})= \emptyset$ for all $1\leq d\leq 4.$ To deal with the case where $E$ has good reduction at a prime over $p$, we develop the \textit{Hecke sieve} to show that all points in $Y_1(m,n)^{(4)}(\F_{p})$ cannot be reductions of points in $Y_1(m,n)^{(4)}(\Q)$ by studying the actions of Hecke operators. We give more details at the beginning of \Cref{sec:rank0}.

To deal with the $\rk J_1(m,n)(\Q)>0$ cases we use the \textit{global method} (see \Cref{prop:hecke_gon}) and its variants to deal with almost all the remaining cases. The idea of this method is that, under some mild assumptions which are satisfied in all our cases, from a hypothetical point of degree $4$ on $Y_1(m,n)$ one could construct a map to $\PP^1$ of degree smaller than known gonality bounds, which is clearly impossible. 

A notable advantage of \Cref{prop_main} and the global method is their minimal reliance on computer calculations, especially in comparison to related works (e.g. \cite{Deg3Class}). These methods significantly reduce the number of cases requiring computational effort. For instance, we handle $X_1(45)$ using \Cref{prop_main} and $X_1(121)$ through the global method, whereas previous approaches demanded extensive computations to address these cases.

%Additionally, our approach circumvents the use of formal immersions, which have been pivotal in classifying torsion groups over number fields of degrees $d = 1, 2, 3$. This difference suggests that our methods may offer greater potential for tackling problems over higher-degree number fields.

%A nice feature of \Cref{prop_main} and the global method is that they require almost no computer calculations, especially when compared to related papers in the field (e.g \cite{Deg3Class}). These two methods allow us to bring down the number of cases where we need to do computer calculations to a minimum. For example, we deal with the case of $X_1(45)$ using \Cref{prop_main}, and $X_1(121)$ using the global sieve, while prior work required extensive computation to deal with these cases (see \cite{Deg3Class}). Furthermore, our approach avoids the use of formal immersions, which have played a significant role in the classification of torsion groups over number fields of degrees $d = 1, 2, 3$. This distinction suggests that our methods may be better suited for addressing the problem in the context of higher-degree number fields.

Another important advantage of our work is that we do not use formal immersions, which have been used extensively in the classification of torsion groups over number fields of degree $d=1,2,3.$  The formal immersion criterion (without additional explicit computations) can only deal with elliptic curves that have bad reduction at all primes in the number field $K$ above some fixed rational prime $p$. So once the order $N$ of the point one wants to rule out is smaller than the Hasse bound at $2$ (in the best case), i.e. smaller than $(2^{d/2}+1)^2$, then the formal immersion approach runs into difficulties. On the other hand, we expect the lower bound for $N$ that the global method can deal with is quasi-linear in the level (more about this will be said in upcoming work by Derickx and Stoll). All this means that our methods are much more suitable to tackling the same question over higher degree number fields. As a demonstration, we apply our methods to higher degree number fields in \Cref{sec:higher_deg}. Specializing to quintic number fields, we eliminate more than half of the torsion groups one would need to eliminate to obtain a complete classification. We do this by "pure thought", i.e., by just checking the LMFDB and performing calculations that can easily be done by hand. 

In \Cref{sec:hecke} we work out a version of the Eichler-Shimura relation for $X_1(2,2n)$. This relation is already well-known for $X_1(n)$. Doing this was essential to be able to apply the Hecke sieve directly on $X_1(2,20)$ and $X_1(2,24)$ to deal with the non-cyclic cases $\tg{2}{20}$ and $\tg{2}{24}$.

To be able to use the Hecke sieve, we implemented the action of Hecke operators on non-cuspidal points on modular curves in \texttt{magma} \cite{magma}.

%\todo[inline]{Say something about other methods if we end up using them.}

\subsection{Overview} \label{sec:overview}
For a finite abelian group $T$ we will say it is a \textit{quartic torsion group} if there exists an elliptic curve $E$ over a quartic field $K$ such that $E(K)_\tor\simeq T$. 

We now consider the groups that need to be eliminated to prove \Cref{main_theorem}. This is the list of all abelian groups $G$ with at most 2 generators such that:
\begin{itemize}
    \item $G$ occurs only finitely often as a quartic torsion group,
    \item for every proper subgroup $H \subsetneq G$, the group $H$ occurs infinitely often as a quartic torsion group,
    \item if $G$ contains a subgroup isomorphic to $(\Z/n\Z)^2$, then $\varphi(n)|4.$
    \item $|G|$ is divisible only by primes $\leq 17.$
\end{itemize}
The groups $G$ that occur infinitely often as quartic torsion groups can be found in \Cref{infinite_quartic}. The third bullet point is necessary because of the Weil pairing, while the fourth follows from \cite[Theorem 1.2.]{DKSS}.

The list of these groups is as follows:
\begin{itemize}
\item $\Z/n\Z, n =  {25^e, 26^e, 27^c, 28^e, 30^e, 32^e, 33^c, 34^e, 35^c, 36^e, 39^c, 40^a, 42^e, 44^a,}$  ${45^c, 48^a, 49^c, 51^c, 55^c, 63^f, 65^f, 77^d, 85^d, 91^d, 119^c, 121^d, 143^d, 169^d, 187^d,}$
${221^d, 289^d, }$

\item $\tg{2}{2n}$, $n=10^e,11^c,12^e,$
\item $\tg{3}{3n}$, $n=4^b,5^f,6^b,7^c,$
\item $\tg{4}{4n}$, $n=3^b,4^b,$
\item $\tg{5}{5n}$, $n=2^b,3^b,$
\item $\tg{8}{8n}$, $n=1^b.$
\end{itemize}

The cases have been marked according to the strategy they have been dealt with:
\begin{itemize}
    \item[$a$] These do not need to be studied separately, since there is $\Q$-rational map $X_1(4n) \to X_1(2,2n)$. So, the approach here is to show that $\tg{2}{2n}$ does not occur for $n=10,11$ and $12$ which is necessary in any case.
     
    \item[$b$] These were already done previously; see \Cref{thm:BN}.
    \item[$c$] These are done in \Cref{cor:easy} using \Cref{prop_main}.

    \item[$d$] These are done in \Cref{prop:global_argument} using a global argument.
    \item[$e$] These are done using the Hecke sieve in \Cref{cor:2_2n}, with an additional argument for $\torz{30}$ in \Cref{prop:30}.
    \item[$f$] These cases are solved using some ad hoc arguments in \Cref{prop:3x15,prop:63,prop:65}.
\end{itemize}

\begin{remark}
    While writing the paper, we became aware that some of the torsion groups $\torz{n}$ for which $\rk J_1(n)(\Q)=0$ have been eliminated, independently of us, by Michael Cerchia and Alexis Newton. The values of $n$ they solved are $n=25, 26, 27, 28, 34, 35, 40$ (private communication).
\end{remark}

\subsection{Data availability and reproducibility}
All the code for the computations in this article is available at: 

\begin{center} \url{https://github.com/nt-lib/quartic-torsion} \end{center}

Claims in this article that are based on computations have been labeled in the source code in order to make it easy to navigate between this text and the code.
 For example, in the proof of \Cref{prop:30} there is a sentence: ``This shows that $z$
    is not the reduction of a point $x^{(4)}\in Y_1^{(4)}(30)(\Q),$ see \github{X1_30.m}{Claims 2,4,5} in the cases where $D$ is of degree $2$ and $3$ or irreducible of degree $4$."
 
 In order to verify, for example, Claim 2, one can open the file \githubbare{X1_30.m} and search for ``Claim 2" to find the following code that supports this claim.
{\small
 \begin{verbatim}
// Claim 2: Cannot have reduction of the form deg2 + 2 cusps
sieved := HeckeSieve(X, q, NoncuspidalPlacesUpToDiamond(X,2));
assert #sieved eq 0;
print "Claim 2 successfully verified";
\end{verbatim}
 }

 To be able to use the Hecke sieve, we implemented the action of Hecke operators on non-cuspidal points on modular curves over finite fields in \texttt{magma} \cite{magma}. The code for computing these Hecke operators and performing the Hecke sieve has been written in a reusable way and has been released as part of mdmagma at \begin{center} \url{https://github.com/koffie/mdmagma}. \end{center} 

The computations have been run on a server at the University of Zagreb with an Intel Xeon W-2133 CPU @ 3.60GHz with 12 cores and 64GB of RAM running Ubuntu 18.04.6 LTS and \texttt{Magma V2.28-3} \cite{magma}.
The computations took approximately 6.5 CPU hours and the maximal memory usage was roughly 1.5GB of RAM. Almost 3 of these 6.5 hours were spent on ruling out the possibility of $\tg{2}{24}$ torsion. The peak RAM usage was in ruling out $\tg{2}{20}$ torsion, with every other computation requiring less than half the amount of RAM.
More detailed logs of the computations are available in the \githubbare{logs} subdirectory of the quartic-torsion repository. \todo{fix this with ANTS template}

 \section*{Acknowledgements}
 We would like to thank Andrew Sutherland for computing the explicit equations for $X_1(m,n)$ that we used, and Barinder S. Banwait, Abbey Bourdon and Michael Stoll for helpful comments on an earlier version of the paper. We are grateful to the anonymous referees for their very helpful comments.

\section{Conventions and notation} 

For a number field $K$, we denote by $O_K$ its ring of integers. The \textit{gonality} $\gon_K X$ of a curve $X$ over $K$ is the least degree of a morphism $X\rightarrow \PP^1$ defined over $K.$

Let $m\mid n$. We define $X_1(m,n)$ to be the moduli space over $\Z[\zeta_m,\frac 1 n]$ parametrising triples $(E,P,Q)$, where $E$ is a generalized elliptic curve, and $P$ and $Q$ are points of order $m$ and $n$ respectively that generate a subgroup of $E[n]$ isomorphic to $\tg{m}{n}$ such that the Weil pairing satisfies $e_m(P, n/mQ) = \zeta_m$. Furthermore, $Y_1(m,n) \subseteq X_1(m,n)$ is defined as the open sub-locus consisting of actual elliptic curves.
If $m=1$ then the Weil-pairing condition is trivial, and if $m=2$ the Weil pairing condition can also be dropped since it follows from the fact that $P,Q$ generate a subgroup isomorphic to $\Z/m\Z \times \Z/n\Z$. The curve $X_1(m,n)$ is a smooth projective curve over $\Z[\zeta_m, \frac{1}{n}]$ with geometrically irreducible fibres. It is a fine moduli space if $n>4$.

After fixing an embedding $\Z[\zeta_m, \frac 1 n] \hookrightarrow \C$ that takes $\zeta_m$ to $e^{\frac{2\pi i} m}$ (which we fix throughout the paper), it makes sense to consider the $\C$-points of $X_1(m,n)$. The $\C$-points $X_1(m,n)(\C)$ form a Riemann surface isomorphic to $\Gamma_1(m,n) \backslash \Hamil^*$, where
\begin{align*}
\Gamma_1(m,n) := \left\lbrace 
\begin{pmatrix}
a & b \\ c & d
\end{pmatrix} \in \SL_2(\Z) \,\middle|\, 
\begin{array}{l}
a \equiv 1 \pmod{m},\quad b \equiv 0 \pmod{m}, \\
c \equiv 0 \pmod{n},\quad d \equiv 1 \pmod{n}
\end{array}
\right\rbrace.
\end{align*}
The isomorphism $X_1(m,n)(\C) \cong \Gamma_1(m,n)/\Hamil^*$ is not unique. In this article we will use the isomorphism that is given on $\Hamil$ by $\tau \mapsto (\C/\Lambda_\tau,\tau/m,1/n)$, see \Cref{rem:H_X_1_iso} for more details on how this choice relates to other choices in the literature. 
Note that the condition $a \equiv 1 \pmod{m}$ is superfluous, since the conditions on $c$ and $d$ imply the stronger congruence $1=ad-bc \equiv a \pmod n$.
For $n>4$, the index of $\pm \Gamma_1(m,n)$ in $\PSL_2(\Z)$ is $$\frac{mn^2}{2}\prod_{p|n}\left(1-\frac 1 {p^2}\right),$$
see e.g. \cite[Table 1]{HamakiotesBau}.

\subsection{Analytic and algebraic ranks of modular abelian varieties}
If $A$ is an abelian variety over a number field $K$ then we use $\rk A(K)$ to denote the rank of the Mordell-Weil group $A(K)$. The integer $\rk A(K)$ is also called the algebraic rank of $A$. Similarly we use $\rk_{an} A(K)$ to denote the analytic rank of $A$ which is defined to be the order of vanishing of $L(A,s)$ at $s=1$. The BSD conjecture predicts that $\rk A(K) = \rk_{an} A(K)$. We use $S_k(\Gamma_1(n))^{new}$ to denote the set of all cuspidal newforms of weight $k$ for the congruence subgroup $\Gamma_1(n)$. If $f \in S_2(\Gamma_1(n))^{new}$ then we denote by $A_f$ the isogeny factor of $J_1(n)$ corresponding to $f$. This $A_f$ is a simple abelian variety over $\Q$ but might not be simple over $\overline \Q$. Note that for two newforms $f \in S_2(\Gamma_1(n))^{new}$ and $f' \in S_2(\Gamma_1(n'))^{new}$ we have $A_f=A_{f'}$ if and only if $n=n'$ and $f$ and $f'$ are Galois conjugates. It is known that every simple isogeny factor of $J_1(n)$ is isogenous to an abelian variety $A_f$ where $f \in S_2(\Gamma_1(n'))$ and $n' | n$. If $f$ is a newform in $S_2(\Gamma_1(n))^{new}$ then we use $L(f, s)$ to denote its $L$-function and define $\rk _{an} f$ as the order of vanishing of $L(f,s)$ at $s=1$. We use $K_f \subset \C$ to denote the subfield generated by its coefficients. The field $K_f$ is known to be a number field isomorphic to $\Q \otimes \End A_f$. There is the following relation between $L$-functions: 
$$L(A_f,s) = \prod_{\sigma: K_f \to \C} L(\sigma(f),s).$$
Assuming the Birch-Swinnerton--Dyer conjecture, the order of vanishing of $L(\sigma(f),s)$ is the same for all choices of $\sigma \in \Hom_\Q (K_f, \C)$. In particular $\rk_{an} A_f(\Q) = [K_f: \Q] \rk_{an} f$.

A part of the BSD conjecture is known to be true for abelian varieties of the form $A_f$ due to the following celebrated result (we specialize by taking $\chi$ to be the trivial character):
\begin{theorem}[{Kato, \cite[Corollary 14.3]{Kato2004}}]\label{thm:Kato}
Let $A$ be an abelian variety over $\Q$ such that there is a surjective
homomorphism $J_1(n)\rightarrow A$ for some $n\geq 1$. Assume $L(A,1)\neq 0$. Then  $A(\Q)$ is finite.
\end{theorem}
This result of Kato generalizes earlier results by Kolyvagin and Logachev for abelian varieties attached to newforms in $S_2(\Gamma_0(n))$.

We use $\overline \Q$ to denote the algebraic closure of $\Q$ inside $\C$. Since the coefficients of newforms all lie in $\overline \Q$ we get an action of $\Gal(\overline \Q/\Q)$ on $S_2(\Gamma_1(n))^{new}$. We use $S_2(\Gamma_1(n))^{new}_{\Gal}:=  S_2(\Gamma_1(n))^{new}/\Gal(\overline \Q/\Q) $ to denote the set of all Galois orbits of newforms under this action. Since the isogeny class of $A_f$ only depends on the Galois conjugacy class of a newform, it also makes sense to define $A_{[f]}$  when $[f] \in S_2(\Gamma_1(n))^{new}_{\Gal}$ is a Galois orbit. In doing so, we get an isogeny decomposition 
$$J_1(n) \sim \bigoplus_{n' | n} \bigoplus_{f \in S_2(\Gamma_1(n'))^{new}_{\Gal}} A_{[f]}^{\sigma_0(n/n')},$$
where $\sigma_0(n/n')$ denotes the number of positive divisors of $n/n'$.

\begin{example}
In this example we explain how to use the LMFDB to explicitly find a large rank 0 quotient of $J_1(N)$. The LMFDB contains an upper bound on the analytic rank of all weight 2 modular forms of level $\leq 1000$. In this article we often need the existence of a large isogeny factor of $J_1(N)$ of algebraic rank 0 for $N<1000$. Such isogeny factors are easy to find using the data in the LMFDB. The \href{https://www.lmfdb.org/ModularForm/GL2/Q/holomorphic/?level_type=divides&level=65&weight=2&showcol=char_order.analytic_rank}{LMFDB data} on the modular forms of weight $2$ and level dividing $65$ shows that the Galois orbit of newforms with LMFDB label $\lmfdbnewform{65}{2}{a}{a}\in S_2(\Gamma_0(65))^{new}$ is the only orbit for which the analytic rank could potentially be positive.  In particular, it follows that every isogeny factor of $J_1(65)/J_0(65)$ has analytic and hence algebraic rank 0. This implies that $J_1(65)/J_0(65)(\Q)$ is a finite abelian torsion group. In the digital form of this article the words \href{https://www.lmfdb.org/ModularForm/GL2/Q/holomorphic/?level_type=divides&level=65&weight=2&showcol=char_order.analytic_rank}{LMFDB data} will always be a hyperlink that brings one to the LMFDB page containing the data that is needed for the argument. Similarly, LMFDB labels of newforms like \lmfdbnewform{65}{2}{a}{a} will also directly link to the LMFDB page on that newform.
\end{example}

\subsection{Hecke operators} \label{sec:hecke}
Let $q\nmid n$ be a prime, and $T_q$ be a Hecke operator on $X_1(n)$ defined as in \cite[Section 3.3]{DiamondIm}:
\begin{equation}\label{hecke_def_1}
    T_q(E,P)=\sum_{G \subseteq E[q] \text{ of order $q$}}(E/G, P \hspace{-0.25cm}\mod G).
\end{equation}
Recall that $T_{q,\F_q}$ acts on $J_1(n)_{\F_q}$ as $\Frob_q+\diamondop{q}_{\F_q,*}\Ver_q$, where Verschiebung $\Ver_q$ is the dual of the Frobenius, see \cite[Eq. (10.2.3)]{DiamondIm}. This is the Eichler-Shimura relation. 

\begin{remark}\label{rem:H_X_1_iso}
The complex curve $\Gamma_1(n) \backslash \Hamil$ has several models over $\Z[1/n]$, and the shape of the Eichler-Shimura relation depends on such a choice. In this article we follow the conventions of \cite[\S 8.2]{DiamondIm}. This means that $Y_1(n)$ denotes the scheme over $\Z[\frac 1 {n}]$ parameterizing pairs $(E,P)$ of an elliptic curve $E$ together with a point $P$ of order $n$, while $Y_\mu(n)$ denotes the scheme parameterizing pairs $(E,\phi)$ where $\phi : \mu_n \to E$ is a closed immersion. The curves $X_1(n)$ and $X_\mu(n)$ denote their compactifications.  As explained in \cite[Var. 10.2.8]{DiamondIm}, the Eichler-Shimura relation for $X_\mu(n)$ is $T_{q,\F_q}=\Ver_q+\diamondop{q}_{\F_q,*}\Frob_q$. This explains the apparent incompatibility of this section with \cite[Prop. 2.3]{DKSS} and the subsequent remark. Indeed, in \cite{DKSS} the authors use $X_1(n)$ to denote what is denoted by $X_\mu(n)$ here and in \cite{DiamondIm}. The main reason for preferring to work with $X_1(n)$ instead of $X_\mu(n)$ in this article is because we do the explicit computations directly in terms of the moduli interpretation. In particular, we do not care about compatibility with the rational structure coming from $q$-expansions of modular forms. On the other hand, in \cite{DKSS} this compatibility with $q$-expansions is essential due to the heavy use of the formal immersion criterion, which is not used in this article.
\end{remark}

For $q\nmid 2n$ a prime, we define the action of $T_q$ on $X_1(2,2n)$ similarly as in \eqref{hecke_def_1}:
\begin{equation}\label{hecke_def_2}
    T_q(E,P,Q)=\sum_{G \subseteq E[q] \text{ of order $q$}}(E/G, P \hspace{-0.25cm}\mod G, Q \hspace{-0.25cm}\mod G).
\end{equation}
%We will need to consider $T_q$ as an operator on $X_1(2,2n)$. 
There is a degree 2 morphism 
$$X_1(4n)\rightarrow X_1(2,2n)=X_1(4n)/\diamondop{2n+1},$$
$$(E,P)\mapsto (E/\diamondop{2nP}, Q, P  \text{ mod}\diamondop{2nP}), $$
$\text{where }Q \text{ generates } E[2] /\diamondop{2nP},$ which is compatible with $T_q$, $\Frob_q$ and $\Ver_q$. It follows that $T_{q,\F_q}$ acts on $J_1(2,2n)_{\F_q}$ and also satisfies the Eichler-Shimura relation (as does $J_1(n)_{\F_q}$) 
\begin{equation}T_{q,\F_q}=\Frob_q+\diamondop{q}_{\F_q,*}\Ver_q, \label{eq:Eichler-Shimura}\end{equation} where $\diamondop{q}$ acts on $X_1(2,2n)$ as 
$$\diamondop{q}(E,P,Q)=(E,qP,qQ)=(E,P,qQ).$$

The Hecke operator $T_q$ on $X_1(n)$ has a well-known description on $\Gamma_1(n)\backslash\Hamil^*$, see \cite[\S 8.3.1]{stein07}. The description given there needs to be modified slightly to get the correct description on $\Gamma_1(2,2n)\backslash\Hamil^*$. Let $\sigma_q \in \SL_2(\Z)$ be such that $\sigma_q \equiv \left(\begin{smallmatrix} q^{-1} & 0 \\ 0 & q \end{smallmatrix}\right) \pmod {2n}$. Define 
\begin{align*}
R_q :=\set{ \sigma_q\begin{pmatrix} q & 0 \\ 0 & 1 \end{pmatrix}} \bigcup \set{\begin{pmatrix} 1 & i-iq \\ 0 & q \end{pmatrix}  \middle|  0 \leq i < q }.
\end{align*}
Then the Hecke operators $T_q$ and $\diamondop{q}$ on $\Gamma_1(2,2n)\backslash\Hamil^*$ are given by 
\begin{equation}
\begin{aligned}\label{eq:hecke_def}
T_q(\tau)&=\sum_{\gamma \in R_q} \gamma (\tau),\\
\diamondop{q}&=\sigma_q(\tau).
\end{aligned}
\end{equation}
Note that the set $R_q$ used to define $T_q$ for $\Gamma_1(2,2n)$ also works for $\Gamma_1(2n)$ and $\Gamma_1(n)$ since it differs from the set in \cite[\S 8.3.1]{stein07} by multiplication by matrices in $\Gamma_1(2n)$.

The following proposition is not used in eliminating any of the groups in the paper, but we believe it is interesting in its own right, and might be useful in future work. 

\begin{proposition} \label{prop:power-eichler-shimura}
Let $m=1$ or $2$ and $m|n$. Let $q\nmid n$ be a prime, $k>0$ an integer, and let $\Delta(n):=(\torz n)^\times/\diamondop{-1}$ denote the group of diamond operators. Then there exists a monic polynomial $\Psi_k(x) \in \Z[\Delta(n)][x]$ of degree $k$ such that $$\Psi_k(T_{q,\F_q})=\Frob_q^k+(\diamondop{q}_{\F_q,*}\Ver_q)^k.$$
%Here we see $\Z[(\torz n)^\times/\pm 1]$ as the group ring generated by the diamond operators $\diamondop{q}_{\F_q,*}$.
\end{proposition}
\begin{proof}
Let $a :=\Frob_q$ and $b := \diamondop{q}_{\F_q,*}\Ver_q$. Then the elementary symmetric polynomials on $a$ and $b$ are $a+b = T_q$ and $ab = q\diamondop{q}_{\F_q,*}$. Then 
\[T_q^k - a^k-b^k = ab\sum_{i=1}^{k-1} \binom k i a^{i-1}b^{k-i-1}.
\]
Now $\sum_{i=1}^{k-1} \binom k i a^{i-1}b^{k-i-1}$ is a symmetric polynomial of degree $k-2$ and can therefore be expressed in terms of $T_q$ and $q\diamondop{q}_{\F_q,*}$, where the power of $T_q$ occurring is at most $k-2$.  In particular, one can write $\sum_{i=1}^{k-1} \binom k i a^{i-1}b^{k-i-1} = \Psi'_k(T_q)$ for some polynomial $\Psi'_k \in \Z[\Delta(n)][x]$ of degree at most $k-2$. So $\Psi_k := x^k -q\diamondop{q}_{\F_q,*}\Psi'_k$ satisfies the proposition.
\end{proof}

\begin{example}
One can take $\Psi_1 =x$, $\Psi_2=x^2-2q\diamondop{q}$ and $\Psi_3=x^3-3q\diamondop{q}x$ in \Cref{prop:power-eichler-shimura}.
\end{example}

For a Hecke operator $T_q$ on $X_1(n)$ or $X_1(2,2n)$ we denote $A_q:=T_q-q\diamondop{q}-1.$
\begin{proposition} \label{prop:kill_tor}
Let $m=1$ or $2$ and $m|n$. If $q\nmid 2n$ is a prime that splits completely in a number field $K$ then
    \begin{equation}\label{eq:kill_tor}
    A_q(J_1(m,n)(K)_{\tor})=0.
\end{equation}
\end{proposition}
\begin{proof}
    Let $P$ be a point of order $k$ on $J_1(m,n)(K)$ and let $\wp$ be a prime of $K$ above $q$. Then $A_q(P)$ is a point of order dividing $k$. By the Eichler-Shimura relation \eqref{eq:Eichler-Shimura} and since $\Ver_q \circ \Frob_q=q$
    we have 
    $$T_{q,\F_q}(P_{\F_q})=\Frob_q(P_{\F_q})+\diamondop{q}_{\F_q,*}\Ver_q(q)=P_{\F_q}+q\diamondop{q} P_{\F_q}.$$
    The statement of the proposition now follows from the fact that reduction mod $\wp$ is injective on $J_1(m,n)(K)_{\tor}$ because $q>2$ and $\wp$ is unramified in $K$.
\end{proof}

\subsection{Cusps on $X_1(m,n)$}
Let $m \mid n$ be integers. The goal of this subsection is to gather some results on the cusps of $X_1(m,n)$.  %that don't seem to be in the literature.
%We restrict to $m=1,2$ since this ensures that $\Gamma_1(m,n)$ has surjective determinant and $X_1(m,n)$ is a curve over $\Q$ so we don't need to worry about the Weil pairing. \todo{check if actually needed}

Cusps of the modular curve $X_1(m,n)$ correspond to \textit{N\'eron} $k$-gons, where $m\mid  k\mid n$, together with a level structure; for the definition see \cite[Sections 1.3 and 2.2.1]{DerickxPhD}. Over an algebraically closed field $K$ a N\'eron $k$-gon is a group scheme isomorphic to $\torz{k} \times \mathbb G_m$. The Weil pairing on $(\torz{k} \times \mathbb G_m)[m] \cong \torz{m} \times \mu_m$ is defined by \begin{align*}
    e_m : (\torz{m} \times \mu_m) \times  (\torz{m} \times \mu_m)  &\to \mu_m, \\
    \left((a,\zeta_m^b), (c,\zeta_m^d)\right) &\mapsto \zeta_m^{ad-bc}.
\end{align*}
If $n$ is invertible in $K$, a level structure for $X_1(m,n)$ on $\torz{k} \times \mathbb G_m$ is a pair of points $P$ and $Q$ in $\torz{k} \times \mathbb G_m$ of order $m$ and $n$ respectively that generate a subgroup isomorphic to $\tg{m}{n}$ such that the Weil pairing satisfies $e_m(P, \frac n m Q)= \zeta_m$ with the additional requirement that the projections of $P$ and $Q$ onto $\torz{k}$ together generate $\torz{k}$.

A divisor is called \textit{cuspidal} if it is supported entirely on cusps, and \textit{non-cuspidal} if its support contains no cusps.

Note that on the non-cuspidal part we have an isomorphism 
\begin{equation}
\begin{aligned}\label{eq:upperhalf_iso}
    \phi: \Gamma_1(m,n)\backslash \Hamil &\to Y_1(m,n)(\C), \\
             \tau & \mapsto (\C/\Lambda_\tau,\tau/m,1/n) , \quad \text{with } \Lambda_\tau := \tau \Z +\Z.
\end{aligned}
\end{equation}
When $m=1$ this simplifies to $\tau \mapsto (\C/\Lambda_\tau,1/n)$ as in \cite[\S 7.2]{DiamondIm}. The left-hand side of \eqref{eq:upperhalf_iso} admits a compactification by replacing the upper half plane with the extended upper half plane $\Hamil^* := \Hamil \cup \mathbb{P}^1(\mathbb{Q})$, while the right-hand side admits a compactification in terms of generalized elliptic curves and N\'eron $n$-gons. 
%Now let us describe how $\phi$ extends on these compactifactions, where we will first focus on $X_1(n)$, i.e. the case $m=1$.

Let $w$ be an integer and $E_{q,w}$ denote the base change to $\C((q^{1/w}))$ of the Tate curve over $\Z[[q]]$; then $E_{q,w}(\C((q^{1/w}))) \cong \C((q^{1/w}))^\times/q^\Z$. Furthermore $E_{q,w}$ has split multiplicative reduction at $q^{1/w}=0$ and the size of the component group is $w$ \cite[Remark IV.9.6]{silverman-advanced}. This means that the special fiber of the N\'eron model is isomorphic to $\torz{w} \times \C^\times$. Additionally, we have the following description of the reduction mod $q^{1/w}$ map
\begin{align*}
    \red_{q^{1/w}} : \C((q^{1/w}))^\times/q^\Z & \to \torz{w} \times \C^\times, \\
                     f(q) & \mapsto \left(a \pmod w, \left(\frac{f(q)}{q^{a/w}}\right)(0) \right), 
\end{align*}
where $a := v_{q^{1/w}}(f)$.
Let $\Gamma \subseteq \SL_2(\mathbb Z)$ be  a congruence subgroup and let $w$ denote the width of the cusp $\infty$. Then $q^{1/w}$ is a local coordinate at $\infty$ on $\Gamma \backslash \Hamil^*$. Furthermore, $E_{q,w}$ (up to quadratic twists) can be thought of as a model of the universal elliptic curve over $\C((q^{1/w}))$, i.e. over a formal neighbourhood of $\infty$.
In terms of the extended complex upper half plane, one has that the cusps are given by $\Gamma_1(m,n) \backslash\PP^1(\mathbb{Q})$. The following gives a more explicit description of this set.

\begin{proposition}\label{prop:cusp_equiv}
Let $s = a/c$ and $s' = a'/c'$ be elements of $\mathbb{Q} \cup \{\infty\}$ with
$\gcd(a, c) = \gcd(a', c') = 1$. Then

\[
\Gamma_1(m,n)s' = \Gamma_1(m,n)s 
\Longleftrightarrow
\begin{pmatrix} a' \\ c' \end{pmatrix} \equiv \pm 
\begin{pmatrix} a + jm c \\ c \end{pmatrix} \!\! \mod{n}
\text{ for some $j\in \Z$}.
\]
\end{proposition}
\begin{proof}
This is a slight generalization of \cite[Lemma 3.8.3]{modular}, and the proof goes through ad verbatim.
\end{proof}
Furthermore, an element $\left(\begin{smallmatrix} \bar a  \\ \bar c  \end{smallmatrix}\right) \in (\Z/n\Z)^2$ has a lift $\left(\begin{smallmatrix}  a  \\  c  \end{smallmatrix} \right) \in \Z^2$ with $\gcd(a,c)=1$ if and only if $\left(\begin{smallmatrix} \bar a  \\ \bar c  \end{smallmatrix}\right)$ has order $n$, see \cite[Proposition 3.8.4]{modular}. This gives a short combinatorial description of the set of cups.

Let $a,c \in \Z$ be coprime integers, and let $a/c \in \PP^1(\Q)$ denote the corresponding cusp. Pick $b,d \in \Z$ such that $ad-bc=1$ and let $\gamma:=\left(\begin{smallmatrix} a & b \\ c & d \end{smallmatrix}\right)$; if $\tau$ goes to $\infty$ then $\gamma \tau$ goes to $a/c$. 
Multiplication by $(c\tau +d)$ induces an isomorphism $\psi_{\gamma} : \C/\Lambda_{\gamma(\tau)} \to \C/\Lambda_\tau$. This shows that 
\begin{align}\label{eq:tau_to_q}(\C/\Lambda_{\gamma \tau},\gamma(\tau)/m,1/n) \cong (\C/\Lambda_\tau,(a\tau+b)/m,(c\tau+d)/n).
\end{align}
Let $w$ denote the width of the cusp $a/c$, i.e. $w$ is the smallest positive integer such that
\begin{align*}
(a\tau+b)/m &\equiv (a(\tau+w)+b)/m &\pmod {\Lambda_\tau},\\
(c(\tau+w)+d)/n&\equiv (c\tau+d)/n &\pmod {\Lambda_\tau}.
    \end{align*}
This is equivalent to $w$ being the smallest positive integer such that both $aw/m$ and $cw/n$ are integers. Since $\C/\Lambda_\tau \cong \C^\times/q$, where $q=e^{2\pi i \tau}$, we have a way to describe what happens with the points $(a\tau+b)/m$ and $(c\tau+d)/n$ of order $m$ and $n$ when reducing modulo $q^{1/w}:=e^{2\pi i \tau/w}$ on the Tate curve. Indeed, by applying the isomorphism $\C/\Lambda_\tau \cong \C^\times/q$ to the right-hand side of \eqref{eq:tau_to_q} one gets that $(E_{q,w},q^{\frac a m}e^{2\pi i \frac b m},q^{\frac c n}e^{2\pi i \frac d n})$ is a model for the universal elliptic curve over $X_1(m,n)$ in a formal neighborhood of the cusp $a/c$. Reducing the two points on $E_{q,w}$ modulo $q^{1/w}$ gives
\begin{equation}
\begin{aligned}\label{eq:cusp_reduction}
%\red_{q^{\frac 1 w}}e^{2\pi i \frac{a\tau+b}m}&=
\red_{q^{\frac 1 w}}q^{\frac a m}e^{2\pi i \frac b m}&=\left(\frac {aw} m \mod w, e^{2\pi i \frac b m}\right), \\
%\red_{q^{\frac 1 w}}e^{2\pi i \frac{c\tau+d}n}&=
\red_{q^{\frac 1 w}}q^{\frac c n}e^{2\pi i \frac d n}&=\left(\frac {cw} n \mod w, e^{2\pi i \frac d n}\right),
\end{aligned}
\end{equation}
where we note that $aw/m$ and $cw/n$ are integers by the definition of $w$.

The above description of how $\phi$ gets extended to the cusps allows us to prove the following.
%Indeed $e^{2\pi i (c\tau +d)/n}=q^{c/n}e^{2\pi i d/n}$
%Let $p$ be a prime coprime to $m$,$\sigma_p$

%$$T_p( \C/\Lambda_\tau,\tau/m,1/n)=(\C/\Lambda_{p\tau},p\tau/m,p/n)+\sum_{i=0}^{p-1} ( \C/\Lambda_{(\tau+i)/p},\tau/m,1/n)$$

\begin{lemma}\label{lem:special_cusp}
Let $m=1$ or $2$, $n$ an integer such that $m|n$ and $p$ a prime that does not divide $2n$.  Let $C_0 \in X_1(m,n)(\mathbb \Z[1/n])$ be the cusp whose moduli interpretation is the  N\'eron $n$-gon $E_0 := \torz{n} \times \mathbb G_m$ together with the points $P_0:=(0,(-1)^{m-1})$ of order $m$ and $Q_0 :=(1,1)$ of order $n$. Then $C_{0,\C}=\phi(0)$ under the extension of $\phi$ from \cref{eq:upperhalf_iso} to $\Hamil^*$. Furthermore, $A_p(C_0)=0$.
\end{lemma}
\begin{proof}
\begin{comment}
First, we find $a/c \in \P^1(\Q)$ such that $\phi(a/c)=C_0$, for this we need $w=n, cw/n=c \equiv 1 \pmod n, d\equiv0 \pmod n, a=0\mod m, b=1\mod m$. This can be achieved by taking $a=0,c=1,b=-1,d=0$ which corresponds to the cusp $0 \in \P^1(\Q)$.
\end{comment}
We write $0=a/c$ with $a=0$ and $c=1$ and apply the above description of $\phi$ when extended to the cusps. We choose $b=-1$ and $d=0$, so that $ad-bc=1$. It is clear that $w=n$ is the smallest positive integer such that $aw/m=0$ and $cw/n=w/n$ are integers. From \cref{eq:cusp_reduction} we get $$\phi(0) =(\Z/n\Z \times \mathbb G_{m,\C}, (0,e^{2\pi i \frac {-1} 2}),(1,1)) = (E_{0,\C},P_{0,\C}, Q_{0,\C})=C_{0,\C}.$$
We notice that it suffices to prove $A_p(C_0)=0$ over $\C$ as the divisor $A_p(C_0)$ has no vertical components. If $\tau \in \Hamil^*$ we use $\set{\tau}$ to denote $\tau$ seen as a divisor on $\Gamma_1(m,n)\backslash \Hamil^*$. Using the isomorphism $\phi$, we are reduced to showing $A_p(\set{0})=0$. We do this using \cref{eq:hecke_def} and \Cref{prop:cusp_equiv}:
\begin{align*}
 T_p\set{0}  &= \sum_{\gamma \in R_p} \gamma\set{0} = \sigma_p \set{0} + \sum_{i=0}^{p-1} \set{(i-pi)/p}\\
 &= \sigma_p \set{0} +\set{0} + (p-1)\sigma_p\set{0}=p\diamondop{p}\set{0}+{0},
\end{align*}
where we use \Cref{prop:cusp_equiv} to show that $\set{(i-pi)/p}=\sigma_p\set{0}$ when $p \nmid i$. Indeed if $p\nmid i$ then $i-pi$ and $p$ are coprime, so $\set{(i-pi)/p}=\sigma_p\set{0}$ follows from the congruence
$$\begin{pmatrix}i(1-p)\\p\end{pmatrix}\equiv \sigma_p \begin{pmatrix}pi(1-p)\\1\end{pmatrix}$$
and the fact that $m\mid 1-p$.

%As before we let $\gamma:=\left(\begin{smallmatrix} a & b \\ c & d \end{smallmatrix}\right)=\left(\begin{smallmatrix} 0 & -1 \\ 1 & - \end{smallmatrix}\right)$, so that if $\tau$ goes to $\infty$ then $\gamma \tau$ goes to the cusp $0$. Using $\psi_\gamma$ we get an isomorphism $(\C/\Lambda_{\gamma \tau},\gamma(\tau)/m,1/n) \cong (\C/\Lambda_\tau,-1/m,\tau/n).$ This isomorphism allows us to describe the hecke operator $T_p$ in a neighborhood of the cusp $0$ using the Tate curve. Indeed,
%\begin{align}
%    T_p (\C/\Lambda_\tau,-1/m,\tau/n) = (\C/\Lambda_{p\tau},-p/m,p\tau/n)+\sum_{i=0}^{p-1} (\C/\Lambda_{(\tau+i)/p},-1/m,\tau/n)
%\end{align}

\end{proof}

\section{Previous results}
We recall previously proven results about torsion groups of elliptic curves over quartic fields. The following result tells us which quartic torsion groups appear infinitely often (in the sense that there exist infinitely many elliptic curves with different $j$-invariants with that torsion group over some quartic field). 

\begin{theorem}[{\cite[Theorem 3.6.]{JeonKimPark06}}]\label{infinite_quartic}
If $K$ varies over all quartic number fields and $E$ varies over
all elliptic curves over $K$, the groups that appear infinitely often as
$E(K)_{\tor}$ are exactly the following
\begin{alignat*}{2}
& \torz{n},  && n=1-18, 20,21,22,24 \\
& \tg{2}{2n}, \quad \quad && n=1-9 \\
& \tg{3}{3n}, \quad && n=1-3 \\
& \tg{4}{4n}, \quad && n=1,2 \\
& \tg{5}{5}, \quad && \\
& \tg{6}{6}. \quad && \\
\end{alignat*}
\end{theorem}

Some possible torsion groups have been ruled out in previous work of Bruin and Najman \cite{bruin_najman2016}, as can be seen from the following result.

\begin{theorem} \label{thm:BN}
  The following groups do not appear as $E(K)_{\tor}$ for any elliptic curve $E$ over any quartic field $K$:
\begin{alignat*}{2}
& \tg{3}{3n}, \quad && n\in \{4,6,9,11,13\}, \\
& \tg{4}{4n}, \quad && n\in \{3,4,7,11,13,17\}, \\
& \tg{5}{5n}, \quad && n\in \{2,3\}, \\
& \tg{8}{8}. \quad && \\
\end{alignat*}
\end{theorem}
\begin{proof}
    All the cases follow from {\cite[Theorem 8]{bruin_najman2016}}, except from $\tg{5}{5n}$ for $n=2$ and $3.$ These could a priori be torsion groups for elliptic curves only over $\Q(\zeta_5)$, but {\cite[Theorem 6]{bruin_najman2016}} shows that this is impossible.
\end{proof}

\section{Auxiliary results}
In this section we list (and slightly modify) some known results that we will use. 

\begin{theorem}[Abramovich's bound]
\label{abramovich}
    $$\gon_\C X_\Gamma>\frac{325}{2^{15}}[\PSL_2(\Z): \Gamma].$$
\end{theorem}
\begin{proof}
    This is \cite[Theorem 0.1]{abramovich} using the bound on the leading nontrivial eigenvalue of the non-Euclidean Laplacian $\lambda_1>\frac{975}{4096}$ from \cite[Appendix 2]{Kim03}.
\end{proof}

Let $H \leq (\Z/n\Z)^\times$ be a subgroup containing $-1$. In this paper, by $X_H$ we will denote the modular curve corresponding to 
$$\Gamma_H:=\left\{ \begin{pmatrix}
    a & b\\ 0 & d
\end{pmatrix}\in \GL_2(\torz{n}) \mid  a \in H  \right\}.$$
By $J_H$ we denote the Jacobian of $X_H$. The coarse moduli spaces $X_H$ and $X_{\pm H}$ are isomorphic. This means that the condition $-1 \in H$ is not a genuine restriction and was only included to make the statement of \Cref{prop:hecke_gon} cleaner. The curve $X_1(n)$ is isomorphic to $X_H$ with $H = \set{\pm 1}$ and for the main result of this paper we will only use \Cref{prop:hecke_gon} with $H=\set{\pm 1}$. Although not needed for this paper, the extra flexibility of the choice of the group $H$ in this proposition is useful even if one is only interested in the modular curve $X_1(n)$. Indeed, \Cref{ex:usefullness-of-H} shows that there exist examples where the finiteness condition on $T := (\diamondop{a}-1)(J_H(\Q))$ is satisfied, but where one cannot use the proposition with $H=\{\pm 1 \}$.

\begin{proposition}\label{prop:hecke-kernel}
Let $n$ be an integer, $q \nmid n$ a prime and let $H$ be a subgroup
  of~$(\Z/n\Z)^\times$ containing $-1$. Let $a=\sum_{i=1}^k a_i\diamondop{d_i} \in \Z[(\Z/n\Z)^\times/H]$ be a linear combination of diamond operators. We consider $t :=T_q - a$ 
as a correspondence on $X_H$, inducing an endomorphism of the divisor group of $X_H$
over $\C$. Then the kernel of $t$ is contained in the subgroup of divisors supported at cusps.
\end{proposition}
\begin{proof}
This is a slight generalization of \cite[Proposition 2.4]{DKSS} and the proof carries over ad verbatim.
\end{proof}

The following proposition is a generalization of a proposition of the first author \cite[Proposition 7.1]{DKSS}, and a slight modification of a version of the result that appears in the preprint \cite{derickxstoll} as Proposition 4.1. As \cite{derickxstoll} is still unpublished, we include the proof here as well for completeness.

\begin{proposition}[Derickx] \label{prop:hecke_gon}
  Let $d \ge 1$, let $n$ be an integer, $q \nmid n$ a prime and let $H$ be a subgroup
  of~$(\Z/n\Z)^\times$ containing $-1$. Let $a \in (\Z/n\Z)^\times/ H$ be such that $T := (\diamondop{a}-1)(J_H(\Q))$ is finite.
  When $q=2$ we either assume that $\#A$ is odd or, more specifically, that the $2$-primary part of~$T$ is killed by $A_2$.
  We then set
  \[ k_q = \begin{cases}
           q+1 & \text{if $J_H(\Q)$ is finite, otherwise} \\
           2q + 1 & \text{if $a \in \set{q, q^{-1}}$,} \\
           2(q + 1)& \text{if $a \notin \set{q, q^{-1}}$.}
         \end{cases}
  \] 
  Then $k_qd < \gon_{\Q}(X_H)$ implies that any point
  in~$X_H^{(d)}(\Q)$ without cusps in its support is a sum of orbits
  under~$\diamondop{a}$. In particular, the inequality above holds when
  \begin{equation}\label{DS-inequality}
      d < \frac{325}{2^{15}} \, \frac{[\SL_2(\mathbb Z) :\Gamma_1(n)]}{k_q \cdot \#H}.
  \end{equation} 
\end{proposition}
\begin{proof}
Let $c \in X_1(N)(\Q)$ be a rational cusp such that $A_q(c) = 0$; such a cusp exists by \Cref{lem:special_cusp}.  Let $c'$ be the image of $c$ in $X_H(\Q)$. Define 

\[ B_{a} :=\begin{cases}
           1 & \text{if $J_H(\Q)$ is finite, } \\
           (\diamondop{a}-1) & \text{otherwise.}
         \end{cases}
         \]
Let $D \in X_H^{(d)}(\Q)$ be an effective divisor of degree $d$ without cusps in its support, and consider the linear equivalence class $[D-dc']$ as a rational point in $J_H(N)(\Q)$. Then $A_qB_{a}([D-dc'])$ is $0$. Indeed, if $J_H(\Q)$ is finite then $A_qB_{a}([D-dc'])=0$ because $A_q$ kills all the torsion. If $J_H(\Q)$ is not finite, then $(\diamondop{a}-1)([D-dc'])$ is torsion by assumption and hence also killed by $A_q$.

Now since $A_qc' =0$ it follows that $A_qB_{a}(D-dc')=A_qB_{a}D$.
Now if $[A_qB_{a}D]=0$ there are two possibilities, either $A_qB_{a}D=0$ or there is a non-constant function $f$ such that $A_qB_{a}D = \ddiv f$.

First consider the case $A_qB_{a}D=0$. Since $D$ and hence $B_{a}D$ does not contain any cusps in its support, one can apply \Cref{prop:hecke-kernel} to deduce $B_{a}D=0$. If $J_H(\Q)$ is finite, this implies $D=0$, so in this case there does not exist an effective divisor $D \in X_H^{(d)}(\Q)$ of degree $d$ without cusps in its support, and there is nothing to prove. If $J_H(\Q)$ is not finite, then $B_{a}D=0$ implies $\diamondop a D =D$ forcing $D$ to be a sum of orbits under $\diamondop a$.

What remains to show is that the second case, $A_qB_{a}D = \ddiv f$ for some non-constant function $f$, cannot happen.

By separating the positive and the negative terms one can rewrite $A_qB_{a}D$ as \begin{align}A_qB_{a}D=\begin{cases}
           T_qD - (q\diamondop{q}+1)D  & \text{if $J_H(\Q)$ is finite, } \\
           (T_q\diamondop{a}+q\diamondop{q}+1)D - (T_q +\diamondop{a}(q\diamondop{q}+1))D & \text{otherwise}. 
         \end{cases}\label{eq:linear-equivalence}\end{align}
First consider the case when $J_H(\Q)$ is finite. Let $k_q:=q+1$. Note that $T_qD$  and $(q\diamondop{q}+1)D$ are both effective divisors of degree $k_qd$, so $\deg f \leq k_qd$ which contradicts $k_qd < \gon_\Q(X_H).$

Now consider the case when $J_H(\Q)$ is not finite. Let $k_q:=2(q+1)$.
If $a \notin \set{q,q^{-1}}$ then $(T_q\diamondop{a}+q\diamondop{q}+1)D$ and $(T_q +\diamondop{a}(q\diamondop{q}+1))D$ are effective divisors of degree $k_qd$ and $\deg f \leq k_qd$ again contradicting $k_qd < \gon_\Q(X_H).$
In the cases $a=q$ and $a=q^{-1}$ the right-hand side of \eqref{eq:linear-equivalence} simplifies to $(T_q\diamondop{q}+(q-1)\diamondop{q}+1)D - (T_q +q\diamondop{q^2})D$ respectively $ (T_q\diamondop{a}+q\diamondop{q})D - (T_q +(q-1)\diamondop{q}+\diamondop{q^{-1}})D.$ This simplification decreases the degree of the positive and negative part by 1 and hence there is again a contradiction with $k_qd < \gon_\Q(X_H)$ as before.
         
\end{proof}

\begin{corollary}\label{cor:global}
    Let the assumptions be as in \Cref{prop:hecke_gon}, and assume that the order of $a \in (\Z/n\Z)^\times/ H$ does not divide $d$. Then any rational point 
  on~$X_H^{(d)}$ contains either a cusp or a CM point whose $j$-invariant is $0$ or $1728$ in its support.
\end{corollary}
\begin{proof}
Let $D \in X_H^{(d)}(\Q)$. If $D$ has a cusp in its support, there is nothing to prove. So by \Cref{prop:hecke_gon}, $D$ will be a sum of orbits under $\diamondop{a}$. Since the order of $a$ does not divide $d$, there will be at least one point $P \in X_H(\Qbar)$ in the support of $D$ that is fixed by some power of $\diamondop{a}$ that is not the identity. The result follows, since the fixed points of the diamond operators are cusps or CM points with $j=0$ or $1728$.
\end{proof}

\begin{theorem}[{Waterhouse, \cite[Theorem 4.1]{waterhouse69}}] \label{waterhouse}
Let $E$ be an elliptic curve over a finite field $\F_q$, where $q=p^a$. Let $\beta=q+1-\# E(\F_q)$. 
Then $|\beta|\leq 2\sqrt q$, and $\beta$ satisfies one of the following:
\begin{enumerate}
    \item $(\beta,p)=1$;
    \item $a$ is even and $\beta=\pm 2\sqrt q$;
    \item $a$ is even, $p\not\equiv 1 \pmod 3$ and $\beta=\pm \sqrt q$;
    \item $a$ is odd, $p=2$ or $3$ and $\beta=\pm p^{\frac{a+1}2}$;
    \item $a$ is odd and $\beta=0$;
    \item $a$ is even, $p\not\equiv 1 \pmod 4$ and $\beta=0$.
\end{enumerate}
\end{theorem}

\begin{corollary} \label{cor:curves_over_fp}

    \begin{enumerate}
        \item Suppose $mn$ is odd and $mn>25$, or $mn=15,19$ or $23$. Then $X_1(m,n)$ has no non-cuspidal points over $\F_{2^d}$ for $d \leq 4.$

       \item Let 
        $$S_3=\left\{44,47,67,70,76,79,85,88,94,97\right\}.$$
        Suppose  $m \mid n$, $3\nmid mn$ and either $mn>100$, $mn\in S_3$, or $51 \leq mn\leq 63$. Then $X_1(m,n)$ has no non-cuspidal points over $\F_{3^d}$ for $d \leq 4.$
    \end{enumerate}
\end{corollary}
\begin{proof}
    This follows directly from \Cref{waterhouse}.
\end{proof}

\section{The Hecke sieve for the rank 0 cases} \label{sec:rank0}
Let $E$ be an elliptic curve over a number field $K$ of degree $d$ and $p \in \Z$ be a suitable prime. Our main approach is to split the problem up into cases depending on the reduction of $E$ above the different primes of $K$ above $p$. The argument exposition is tailored to the case in which $\rk J_1(m,n)(\Q(\zeta_m))=0$, but it could be generalized to $\rk J_1(m,n)(\Q(\zeta_m))>0$. We first prove \Cref{prop_main}, which is a modification of \cite[Theorem 1]{bruin_najman2016}, and which allows us to rule out the cases where $E$ has bad reduction at all primes above $p$; which is automatically satisfied if $Y_1(m,n)(\F_{p^d})=\emptyset$ for $1\leq d \leq 4.$  

To deal with the case in which $E$ has good reduction at at least one prime of $K$ over $p$ we use a method that we call the \textit{Hecke sieve} (\Cref{prop:hecke_sieve,cor:bad_reduction}). The Hecke sieve works by studying the point $x^{(d)} \in X_1(m,n)^{(d)}(\Q(\zeta_m))$ corresponding to $\tg{m}{n} \hookrightarrow E(K)$ and using Hecke operators to rule out the different possibilities for $x^{(d)}_{\F_p}$ that have at least one non-cuspidal place in their support. The Hecke sieve has some similarities to the Mordell-Weil sieve, but has the advantage that one does not need to calculate any generators of the Mordell-Weil group $J_1(m,n)(\Q(\zeta_m))$.

\begin{lemma}\label{lem:gon_new}
    Let $X/K$ be a smooth, projective, and geometrically integral curve over a field $K$, $d$ an integer and $K\subseteq K_1\subseteq K_2$ be fields. Let $D_1\in X^{(d)}(K_1)$ and $D_2\in X^{(d)}(K_2)$ be distinct effective divisors. If $[D_2-D_1]=0$ then $\gon_{K_1}X\leq d.$
\end{lemma}
\begin{proof} The divisor $D_2-D_1$ is defined over $K_2$ and is linearly equivalent to zero. So there is a non-constant function $f \in H^0(X_{K_2}, \mathcal O(D_1))$. This shows $\dim_{K_2} H^0(X_{K_2}, \mathcal O(D_1))\geq 2$. Since the dimension of $H^0(X_{K_1}, \mathcal O(D_1))$ is invariant under base change, one has that $\dim_{K_1} H^0(X_{K_1}, \mathcal O(D_1))\geq 2$ as well and hence $\gon_{K_1}X\leq d$.
\end{proof}
Under the assumptions of \Cref{lem:gon_new} it is obvious that $\gon_{K_2}X\leq d$. The lemma is essentially a trick that guarantees the existence of a function of at most degree $d$ over the smaller field $K_1$ as well. Without \Cref{lem:gon_new}, the assumption of $\gon_{\Q(\zeta_m)} X_1(m,n)>d$ would need to be replaced with the stronger assumption $\gon_{\Q(\zeta_n)} X_1(m,n)>d$ in the proposition below.

\begin{proposition}\label{prop_main}
  Let $m\mid n$ be positive integers such that $mn>4$, $p$ be a prime not dividing $mn$, and $d$ an integer. Suppose $\rk J_1(m,n)(\Q(\zeta_m))=0$ and suppose one of the following is true:

\begin{itemize}
    \item[a)] If $p>2$ then $\gon_{\Q(\zeta_m)} X_1(m,n)>d$.

    \item[b)] If $p=2$ then either \begin{itemize}
        \item[i)]  $\gon_{\Q(\zeta_m)} X_1(m,n)>2d$ or
        \item [ii)] $\gon_{\Q(\zeta_m)} X_1(m,n)>d$, and there is a field $L$ with $\Q(\zeta_m) \subseteq L \subseteq \Q(\zeta_n)$ and a prime $\wp_L$ of $L$ above $p$ such that every cuspidal $\F_{\wp_L}$-rational divisor on $X_1(m,n)$ lifts to a $L$-rational cuspidal divisor on $X_1(m,n)$ and reduction mod $\wp_L$ is injective on $J_1(m,n)(L)[2]$. %\maarten{we need this complicated second condition for 27, since we only know that $J_1(27)(\Q)$ contains no $2$-torsion and not that $J_1(27)(\Q(\zeta_{27}))$ contains no $2$-torsion.} 
    \end{itemize}
\end{itemize}  
  Then every elliptic curve $E$ over a degree $d$ extension $K$ of $\Q(\zeta_m)$ with torsion $\tg{m}{n}\subseteq E(K)$ has good reduction at at least one prime of $K$ above $p$.
  
  %  are no elliptic curves with torsion $\tg{m}{n}$ over a degree $d$ extension of $\Q(\zeta_m)$. 

\end{proposition}

\begin{proof}
    Let $\psi:\tg{m}{n} \hookrightarrow E(K)$ denote the inclusion. Throughout this proof $\wp_m$ will denote a fixed choice of prime of $\Q(\zeta_m)$ above $p$, and $\wp_n$ will denote a fixed choice of prime of $\Q(\zeta_n)$ above $\wp_m$. In the case b.ii) we also assume that $\wp_m$ and $\wp_n$ have been chosen compatibly with $\wp_L$, meaning that $\wp_L$ lies above $\wp_m$ and $\wp_n$ lies above $\wp_L$. Finally, $\wp_K$ will denote an arbitrary prime of $K$ above $\wp_m$. We denote by $\widetilde o$ the reduction of an object $o$ modulo $\wp_n$ (if an object is defined over a subfield, then it is base changed to $\Q(\zeta_n)$ and then reduced modulo $\wp_n$ to obtain $\widetilde{o}$).  
    %Let $E$ be an elliptic curve over a number field $K$ of degree $d$ over $\Q(\zeta_m)$ with level structure $\psi:\tg{m}{n} \hookrightarrow E(K)$. 
    Let $x\in Y_1(m,n)(K)$ be a point representing $E$ together with the corresponding level structure $\psi$. Denote the corresponding points $x^{(d)}$ and $\tilde{x}^{(d)}$ in $X_1^{(d)}(m,n)(\Q(\zeta_m))$ and $ X_1^{(d)}(m,n)(\F_{\wp_n})$ respectively.
    
     We first show that $E$ cannot have additive reduction at any prime $\wp_K$ of $K$ above $\wp_m$. Let $\mathcal E$ be the N\'eron model of $E$ over $\OO_K$. Assume for contradiction that $E$ has additive reduction modulo $\wp_K$. Then the group $\mathcal E^0(\F_{\wp_K})$ of $\F_{\wp_K}$-rational points on the identity component $\mathcal E^0$ is of order $\#\F_{\wp_K}$, a power of $p$, and $a:=[\mathcal E(\F_{\wp_K}):\mathcal E^0(\F_{\wp_K})]\leq 4$, so $\#\mathcal E(\F_{\wp_K})=a\cdot \#\F_{\wp_K}$. The reduction mod $\wp_K$ is injective on $\tg{m}{n}\subseteq E(K)[n] = \mathcal E(\OO_K)[n]$, but $mn$ does not divide $a\cdot\#\F_{\wp_K}$, so this is a contradiction. 

    If $E$ has good reduction modulo some prime of $K$ above $p$, there is nothing to prove. So we can assume $E$ has multiplicative reduction at all primes of $K$ above $p$. Then $\tilde{x}^{(d)}$ has to be a sum of reductions modulo $\wp_n$ of $\Q(\zeta_n)$-rational cusps, i.e., $\tilde{x}^{(d)}=\widetilde C$ for some effective cuspidal divisor $C \in X_1^{(d)}(m,n)(\Q(\zeta_n))$.

    Additionally let $c_0 \in X_1(m,n)(\Q(\zeta_m))$ be the cusp whose moduli interpretation is that of a N\'eron $n$-gon and $\tg{n}{m} \to \Z/n\Z\times \mathbb G_m(\Q(\zeta_m))  $ is given by $(a,b) \mapsto (a,\zeta_m^b)$.

    We have
    $$[x^{(d)}-C]=[x^{(d)}-dc_0]+[dc_0 -C].$$
    Since $[x^{(d)}-dc_0]\in J_1(m,n)(\Q(\zeta_m))=J_1(m,n)(\Q(\zeta_m))_{\tor}$ and since $dc_0 -C$ is a cuspidal divisor, by Manin-Drinfeld $[dc_0 -C]$ is torsion, so we conclude that $[x^{(d)}-C]\in J_1(m,n)(\Q(\zeta_n))_{\tor}$. We emphasize here that although we want a result over $\Q(\zeta_m)$, we need to temporarily switch to working over $\Q(\zeta_n)$ over which the divisor $C$ is defined.
    
    Suppose first $p>2$. Reduction mod $\wp_n$ is injective on $J_1(m,n)(\Q(\zeta_n))_{tors}$ by \cite[Appendix]{katz81}, since $p$ is unramified in $\Q(\zeta_n)$, so it follows that $[x^{(d)}-C]=0$. By \Cref{lem:gon_new} it follows that $\gon_{\Q(\zeta_m)}X_1(m,n)\leq d$, which is a contradiction if a) is satisfied.
    
    %Hence we have 
    %$$\dim_{\Q(\zeta_m)} H^0(X_1(m,n)_{\Q(\zeta_m)}, \mathcal O(x^{(d)}))=\dim_{\Q(\zeta_n)} H^0(X_1(m,n)_{\Q(\zeta_n)}, \mathcal O(x^{(d)}))\geq 2.$$
    %This inequality contradicts $\gon_{\Q(\zeta_m)} X_1(m,n)> d$, completing the proof if a) is satisfied.

    Suppose now that $p=2$ and $\gon_{\Q(\zeta_m)} X_1(m,n)> 2d$. By \cite[Prop 2.4]{Parent00} and the fact that $p$ is unramified in $\Q(\zeta_n)$, the kernel of the reduction modulo a prime above $2$ of $J_1(m,n)(\Q(\zeta_n))$ is either trivial or a group of exponent $2$.
    
 Hence it follows that, since $[\tilde{x}^{(d)} - \widetilde{C}]=0$, the divisor class $[x^{(d)}-C]\in J_1(m,n)(\Q(\zeta_n))$ has to be contained in $J_1(m,n)(\Q(\zeta_n))[2]$. Hence $2[x^{(d)}-C]=0$, so by the same argument as above, we arrive at a contradiction with $\gon_{\Q(\zeta_m)} X_1(m,n)> 2d$, completing the proof if b.i)  is satisfied.

    Suppose finally that $p=2$ and that the assumptions of b.ii) are satisfied. By assumption, we can choose $C$ to be $L$-rational, so $[x^{(d)}- C]$ is a $L$-rational divisor as well. Since $\tilde{x}^{(d)}=\widetilde C$ it follows that $[\tilde{x}^{(d)}-\widetilde C]=0$. By the injectivity of the reduction mod $\wp_L$ on $J_1(m,n)(L)_\tor$ (which follows from \cite[Appendix]{katz81} and our assumptions), it follows that  $[x^{(d)}- C]=0$. By the same argument as above, we arrive at a contradiction with $\gon_{\Q(\zeta_m)} X_1(m,n)> d$, completing the proof if b.ii) is satisfied.    
\end{proof}

The following lemma will be useful when applying \Cref{prop_main}.
\begin{lemma}\label{lem:main}
Let $m|n$ be integers and $p$ a prime not dividing $n$, and let $E$ be an elliptic curve over a number field $K$ of degree $d$ such that $\tg{m}{n}\subseteq E(K)$. Suppose that for all $d'\leq d$ there does not exist an elliptic curve $E'/\F_{p^{d'}}$ with a subgroup of $E'(\F_{p^{d'}})$ isomorphic to $\tg{m}{n}$. Then $E$ does not have good reduction at any prime of $K$ above $p$.
\end{lemma}
\begin{proof}
    This follows immediately from the fact that reduction modulo a prime above $p$ is injective on all the torsion of order coprime to $p$.
\end{proof}

\begin{lemma}\label{lem:cusp_action}
Let $m=1$ or $2$ and $m|n$, $q$ a prime not dividing $2n$. Suppose $\rk J_1(m,n)(\Q)=0$. 
Then for any $x^{(d)}\in X_1(m,n)^{(d)}$ we have $[A_q(x^{(d)})]=0$.
\end{lemma}
\begin{proof}
    Let $C_0\in X_1(m,n)$ be a cusp such that $A_q(C_0)=0$; it exists by \Cref{lem:special_cusp}. Let $y=[x^{(d)}-dC_0]\in J_1(m,n)(\Q)$. By \Cref{prop:kill_tor} we have 
    $$0=[A_q(y)]=[A_q(x^{(d)})]-[A_q(C_0)]=[A_q(x^{(d)})].$$
\end{proof}

\begin{proposition}[Hecke Sieve] \label{prop:hecke_sieve} Let $m=1$ or $2$ and $m|n$, $p$ a prime not dividing $n$, $q$ a prime with $q\nmid 2pn$ and $1 \leq d'\leq d$ be integers. Assume $\rk J_1(m,n)(\Q)=0$. Suppose that 
\begin{itemize}

    \item[i)] $D\in Y_1(m,n)^{(d')}(\F_p)$ is a non-cuspidal effective divisor of degree $d'$ such that
    $$[A_{q}(D)]\neq 0,$$ and

    \item[ii)] $C \in X_1(m,n)^{(d-d')}(\F_p)$ is a cuspidal effective divisor of degree $d-d'$, such that $[A_{q}(C)]= 0$,
\end{itemize}
where $A_q=T_q-\diamondop{q}q-1$ as in \cref{eq:kill_tor}. Then a point $z \in Y_1(m,n)^{(d)}(\Q)$ cannot reduce to $D+C \in X_1(m,n)^{(d)}(\F_p)$.
\end{proposition}
\begin{proof} We have that 
$$[A_q(z)]_{\F_p}=[A_q(C)]+[A_q(D)]\neq 0,$$
while on the other hand 
$$[A_q(z)]=0$$
by \Cref{lem:cusp_action}. This is obviously a contradiction since $0\in J_1(m,n)(\Q)$ does not reduce modulo $p$ to a non-zero point in $J_1(m,n)(\F_p)$. 
\end{proof}
\begin{remark}
The only essential property of $A_q$ used in the proof of \Cref{prop:hecke_sieve} is that when $\rk J_1(m,n)(\Q)=0$ then $A_q(J_1(m,n)(\Q))=0$. In the case $\rk J_1(m,n)(\Q)>0$ one can still use the Hecke sieve by replacing $A_q$ with another Hecke operator $t$ such that $t(J_1(m,n)(\Q))=0$
\end{remark}

\begin{corollary} \label{cor:bad_reduction} Let $m=1$ or $2$ and $m|n$, and $p$ a prime not dividing $n$. Suppose $\rk J_1(m,n)(\Q)=0$. Suppose that
\begin{itemize}

    \item[i)] for all $1 \leq d'\leq d$ and all $D\in Y_1(m,n)^{(d')}(\F_p),$ there exists a prime $q_D\nmid 2pn$ such that  
    $$[A_{q_D}(D)]\neq 0,$$
    and

    \item[ii)] for all $d'\leq d$ such that $ Y_1(m,n)^{(d')}(\F_p)\neq \emptyset$ and all effective cuspidal divisors $C$ over $\F_p$ of degree $d-d'$, we have  $[A_{q_D}(C)]= 0$ for all $q_D$ used in step i) with $\deg D=d'$,
\end{itemize}
where $A_q=T_q-\diamondop{q}q-1$. Then every elliptic curve $E$ over a degree $d$ number field $K$ with torsion $\tg{m}{n}\subseteq E(K)$ has bad reduction at all primes of $K$ above $p$.
\end{corollary}
\begin{proof}
Let $x^{(d)}\in Y_1(m,n)^{(d)}(\Q)$ be the effective degree $d$ divisor corresponding to $E$ such that $\tg{m}{n}\subseteq E(K)$. Then $x^{(d)}_{\F_p}$ can be decomposed as $x^{(d)}_{\F_p}=C+D$ where $C$ is the sum of cusps and $D$ is a sum of non-cusps. If $E$ has good reduction at a prime above $p$, then $\deg D \geq  1$. However, by \Cref{prop:hecke_sieve} this cannot happen since there is no $z \in Y_1(m,n)^{(d)}(\Q)$ specializing to $x^{(d)}_{\F_p}=C+D$.
\end{proof}

\begin{lemma} \label{lem:split_completely_cusp}
     Let $m=1$ or $2$ and $m|n$, $q$ a prime not dividing $2n$. Let $d$ be an integer, $K$ be a number field and $C \in X_1(m,n)^{(d)}(K)$ be a cuspidal divisor. Assume that $q$ splits completely in $K$. Then 
     $$[A_q(C)]=0.$$
     \end{lemma}
\begin{proof}
    %The condition that $q$ splits completely in $K$ implies that all $K$-rational cusps reduce to $\F_q$-rational cusps. %It follows by our assumptions on $q$ that 
    %$$J_1(m,n)(K)_\tor \hookrightarrow J_1(m,n)(\F_q).$$
    %Now applying \Cref{prop:kill_tor} shows that $A_q(J_1(m,n)(\F_q))=0$. 
    By \Cref{prop:kill_tor} we have $A_q(J_1(m,n)(K)_\tor)=0$. Let $C_0$ be the cusp defined as in \Cref{lem:special_cusp}. We have $[C-dC_0]\in J_1(m,n)(K)_\tor$ by Manin-Drinfeld, so $[A_q(C-dC_0)]=0$. As $[A_q(dC_0)]=0$ by \Cref{lem:special_cusp}, it follows that $[A_q(C)]=0.$
\end{proof}

We are ready to start eliminating possible quartic torsion groups.

\begin{corollary}\label{cor:easy}
    The groups $\torz{n}$ for $n\in\{27, 33,35,39 ,45,49,51,55,119\}$, $\tg{2}{22}$ and $\tg{3}{21}$ are not quartic torsion groups. 
\end{corollary}
\begin{proof}
    For all the torsion groups in the statement of the corollary, and the primes $p$ which we will choose for each group, it follows by \Cref{cor:curves_over_fp} that, for all $k\leq 4$, there are no elliptic curves containing that torsion group over $\F_{p^k}$.
    For all $\torz{n}$ we consider, it follows that $\rk J_1(n)(\Q)=0$ by \cite[Theorem 3.1]{Deg3Class}, while $\rk J_1(m,n)(\Q)=0$ follows from \cite[Theorem 4.1]{DerickxSutherland} for the groups $\tg{m}{n}$ with $m>1$ that we consider.
    
    For $\torz{n}$, with $n=33,35,39 ,45,49,51,55,119$ we apply \Cref{lem:main} and \Cref{prop_main} case b.i), using $p=2$; the necessary assumption $\gon_\Q X_1(n)>8$ follows from \cite[Theorem 3]{derickxVH}. 
    
    For $\torz{27}$ we apply \Cref{lem:main} and \Cref{prop_main}, case b.ii) with $L=\Q$ and $p=2$, using the fact that $J_1(27)(\Q)$ has no $2$-torsion (see \cite[Table2]{Deg3Class}). To show that all $\F_2$-rational cuspidal divisors lift to $\Q$-rational cuspidal divisors, we note that all cusps of $X_1(27)$ are defined over subfields of $\Q(\zeta_{27})$, in which 2 is completely inert. We have that $\gon_\Q X_1(27)>4$ follows from \cite[Theorem 3]{derickxVH}.

    In the case $\tg{2}{22}$ we apply \Cref{lem:main} and \Cref{prop_main} case a) with $p=3$. The fact that $\gon_\Q X_1(2,22)>4$ follows from \cite[Theorem 2.8]{JeonKimPark06}.

    For $\tg{3}{21}$ we apply \Cref{lem:main} and \Cref{prop_main}, case b.i, with $p=2$, to show that there are no degree $2$ points over $\Q(\zeta_m)$ (and hence no quartic points over $\Q$) on $X_1(3,21)$. We see that $\rk J_1(3,21)(\Q(\zeta_3))=0$ from \cite[Theorem 4.1]{DerickxSutherland} and use \Cref{abramovich} to get
    $$\gon_{\Q(\zeta_m)} X_1(3,21)\geq\gon_{\C} X_1(3,21) \geq 6,$$
    completing this case.

    %Furthermore the prime $2$ is completely inert in $\Q(\zeta_{27})$ so the splitting condition is satisfied. 

    %The prime $2$ is inert in $\Q(\zeta_3)$ and splits into primes of degree $8$ in $\Q(\zeta_{17})$ and $\Q(\zeta_{51})^+$.
\end{proof}

\begin{lemma}\label{lem:CC} Let $X:=X_1(m,n)$, with $m=1$ or $2$ and $m|n$, let $p$ and $q$ be different primes not dividing $2n$ and $k\geq 1$. Suppose that for all cusps $c \in X_1(m,n)(\overline \Q)$ one has that at least one of the following holds:
\begin{itemize}
    \item[a)] at least one prime $\frak p$ of $\Q(c)$ lying over $p$ has inertia degree $>k$,
    \item[b)] there are primes $\frak p, \frak q$ of  $\Q(c)$ lying over $p$ and $q$ such that the decomposition subgroup $D_\frak q$ is a subgroup of that of $D_\frak p$.
\end{itemize}
    Then for all $k'\leq k$ and all cuspidal divisors $C\in X^{(k')}(\F_p)$ one has $[A_q(C)]=0$.
\end{lemma}
\begin{proof}
If $c \in X_1(m,n)(\overline \Q)$ is a cusp then $\Q(c) \subseteq \Q(\zeta_n)$, so $\Q(c)$ is an abelian extension of $\Q$ in which $p$ and $q$ are unramified. In particular, the decomposition subgroups $D_\frak p$ and $D_\frak q$ at $\frak p$ and $\frak q$ do not depend on the choice of primes $\frak p, \frak q$ dividing $p$ and $q$ and $\#D_\frak p$ is the inertia degree at $\frak p$. 
By \cite[Lemma 2.8]{Deg3Class} the irreducible component of the cuspidal subscheme of $X_1(m,n)_{\Z[\frac 1 {2n}]}$ on which $c$ lies is isomorphic to $\Spec \OO_{\Q(c)}[\frac 1 {2n}]$.
Since every cuspidal divisor $C\in X^{(k')}(\F_p)$ is a positive linear combination of irreducible divisors, it is enough to prove the claim for the case where $C$ is irreducible.

Let $c_0 \in X_1(m,n)(\overline\F_p)$ be a point in the support of $C$ and $c \in X_1(m,n)(\overline \Q)$ a lift of $c_0$. Let $\frak p \in \Q(c)$ be a prime lying over $p$ such that $c_0= c_\frak{p}$, define $K:=\Q(c)^{D_\frak p}$ and let $\frak{p'}$ be the prime of $K$ which lies under $\frak p$. The inertia degree of $\frak{p'}$ is $1$ by the definition of $K$ and hence we have $\F_{\fp'}=\F_p$.
Since $C$ is irreducible we have $C = \sum_{ {\sigma\in \Gal(\F_{\fp}/\F_{p})}} \sigma(c_0)$ and hence $k'=\#D_{\frak p}$. 
Define $C':=\sum_{\sigma\in D_{\frak p}} \sigma(c)$, then $C' \in X_1(m,n)^{(k')}(K)$ and we have an equality $$C'_{\frak p'} = \sum_{\sigma\in D_{\frak p}} \sigma(c)_{\frak p}= \sum_{\sigma\in \Gal(\F_{\fp}/\F_{p})} \sigma(c_\frak p) = \sum_{\sigma\in \Gal(\F_{\fp}/\F_{p})} \sigma(c_0)=C$$ of divisors in $X_1(m,n)^{(k')}(\F_p)$.
This equality implies that it suffices to show that $[A_q(C')]=0$. Since $\#D_\frak{p}=k'\leq k$ is the inertia degree of $\frak p$, assumption $a)$ cannot hold for the cusp $c$, and hence by assumption $b)$ we have $D_{\frak q} \subseteq D_{\frak p}$. This implies that $q$ is completely split in $K$ and therefore $[A_q(C')]=0$ by \Cref{lem:split_completely_cusp}.
\end{proof}

We now apply the Hecke sieve to several torsion groups.

\begin{proposition}\label{cor:2_2n}
    The groups $\torz{n}$ for $n=25,26,28,32,34,36,42$, $\tg{2}{20}$ and $\tg{2}{24}$ are not quartic torsion groups.
\end{proposition}
\begin{proof}

We apply the Hecke Sieve to show that these groups are not quartic torsion groups. Let $T$ be one of the torsion groups from the statement of the proposition. For all the modular curves $X:=X_1(n)$ or $X_1(2,2n)$ (and defining similarly $Y:=Y_1(n)$ or $Y_1(2,2n)$) corresponding to our groups, we have $\gon_\Q X>4$ by \cite[Theorem 2.6 and Theorem 2.8]{JeonKimPark06} and $\rk J(X)(\Q)=0$ by \cite[Theorem 4.1]{DerickxSutherland} and \cite[Theorem 3.1]{Deg3Class}.

Let $p$ be a prime not dividing $\#T$, and hence a prime of good reduction of $X$. An elliptic curve $E$ over a quartic field $K$ with torsion $T \hookrightarrow E(K)$ has good reduction at a prime $\wp$ of $K$ over $p$ by \Cref{prop_main} a).

So for the remainder of the proof we consider the case that $E$ has good reduction at at least one prime $\wp$ of $K$ over $p$. We use $x \in X(K)$ to denote the point corresponding to $E$ with the inclusion $T \hookrightarrow E(K)$, and let $x^{(4)} \in X(\Q)$ denote the effective divisor which is the sum of the Galois conjugates of $X$. The fact that $E$ has good reduction at $\wp$ is equivalent to $x^{(4)}_{\F_p}$ being of the form $D+C$ where $D \in Y^{(d')}(\F_p)$ is an effective divisor of degree $d'$ with $1 \leq d' \leq 4$ and $C$ a, possibly trivial, effective divisor supported on the cusps. 

We will complete the proof by verifying that for every $T$ we can find $p$ and $q$ such that the assumptions of \Cref{cor:bad_reduction} are satisfied. This will imply that $E$ has bad reduction modulo $\wp$, which will be a contradiction. In \Cref{table:Hecke_sieve} below we list the choices we make for our computations that verify this. 

In the first column we list the torsion groups $T$. In the second column we give the prime $p$ we use for this torsion group. In the columns $\deg d'$, for $d'=2,3,4$ we list the primes $q$ that we used to deal with that degree, if there is no value of $q$ listed this means $Y(\F_{p^{d'}})=\emptyset$. Note that  the column for $d'=1$ is omitted since for the $p$ we used we always have $Y(\F_{p})=\emptyset$. The irreducible divisors in $Y^{(d')}(\F_p)$ correspond to Galois orbits of points in $Y(\F_{p^{d'}})$ that do not come from a subfield. However $Y^{(d')}(\F_p)$ might also contain reducible divisors. Since $Y(\F_{p})=\emptyset$, the only way this can happen is if $d'=4$ and $D \in Y^{(4)}(\F_p)$ is the sum of two irreducible effective divisors of degree 2. This case is listed in the $\deg 2+2$ column. An absence of $q$ in this $\deg 2+2$ column denotes that $Y(\F_{p^2})=\emptyset$. If $Y(\F_{p^{d'}}) \neq \emptyset$ we explicitly check that $[A_q(D)]\neq 0$ for all irreducible effective divisors $x^{(d')}_{\F_p}\in Y^{(d')}(\F_p)$ in order to verify that \Cref{cor:bad_reduction} i) is satisfied. The value $q$ used  (or multiple values $q$ if more than one was used) is listed in the $\deg d'$ column. To explain the $\deg 2+2$ entry, note that we need to rule out the possibility that $x^{(4)}_{\F_p}\in Y^{(4)}(\F_p)$ is the sum of $2$ degree 2 effective irreducible divisors, so we check $[A_q(D)]\neq 0$ for all such divisors. For each $X_1(n)$ and $X_1(2,2n)$ one can find the computations we performed in our GitHub directory in the corresponding file \verb|X1_n.m| or \verb|X1_2_2n.m|.

It remains to verify \Cref{cor:bad_reduction} ii); this will in some cases impose congruence conditions on the choice of $q$ we are allowed to use. The only time we need to check anything is when $x^{(4)}_{\F_p}$ splits as $x^{(4)}_{\F_p}=D+C$, where $D$ is an irreducible effective non-cuspidal divisor and $C$ is an effective cuspidal (nontrivial) divisor. As $Y(\F_p)=\emptyset$ in all our cases, there are 2 cases we need to consider: when $\deg D=2$, and when $\deg D=3$. 

In the CC column in our table, a congruence condition on $q$ is listed, which guarantees that the assumptions of \Cref{lem:CC} are satisfied, and that we can conclude $[A_q(C)]=0$. %If $\deg D=3$, this means that $q$ splits completely over the fields of definition of all cusps that reduce to cusps defined over $\F_p$. If $\deg D=2$, this means that $q$ splits completely in the field of definition of all cusps of degree $1$ or $2$ on $X_{\F_p}$ which are reductions of cusps of higher degree on $X$. 
We verify these claims in \github{congruence_checks.m}{Claim 1} for $\deg D=3$ and in (see \github{congruence_checks.m}{Claim 2}) for the $\deg D=2$ case.

\begin{table}[H]
\centering
\begin{tabular}{|c|c|c|c|c|c|c|}
\hline
$(m,n)$ & $p$ & \multicolumn{5}{|c|}{$q$} \\
\hline
& &  CC & $\deg2$ &   $\deg 3$ & $\deg 2+2$ & $\deg 4$  \\
\hline
$(1,25)$ & $3$ &    &    &   &   & $7$ \\
\hline
$(1,26)$ & $7$ & & $3$ &  $3$ &  $3$     & $3$ \\
\hline
$(1,28)$ & $5$ & $ 1 \pmod 4$ & $13$  & $13$ &  $3$ &   $3$ \\
\hline
$(1,32)$ & $3$ & &  &    $5$  &  & $5$ \\
\hline
$(1,34)$ & $3$ &  &   &   &    & $5$ \\
\hline
$(1,36)$ & $5$ &  $ 1 \pmod 4$ &   &$13$ & & $7$ \\
\hline
$(1,42)$ & $11$ & & $5$ & $5$ & $5$ &     $5$ \\
\hline
$(2,20)$ & $7$ & & $3$  & $3$  &  $3$    & $3$ \\
\hline
$(2,24)$ & $5$ & $ 1 \pmod 4$ &   &     $17$ & &  $7,11$ \\
\hline

\end{tabular}
\caption{}\label{table:Hecke_sieve}
\end{table}
\end{proof}

\begin{remark}
    In the proof of \Cref{cor:2_2n} we used explicit equations for $X_1(2,2n)$. We used optimized models obtained using the methods of \cite[\S 3]{sutherlandtorsion} to make the computation finish in a reasonable amount of time.
\end{remark}

The case of the torsion group $\torz{30}$ requires a different argument. The following proof was suggested to us by the anonymous referee, and is much simpler than our original proof. 

\begin{proposition} \label{prop:30}
    The group $\torz{30}$ is not a quartic torsion group. 
\end{proposition}
\begin{proof}
    We have $\rk J_1(30)(\Q)=0$ by \cite[Theorem 3.1]{Deg3Class} and that $J_1(30)(\Q)$ is generated by rational cuspidal divisors by \cite[Corollary 4.14]{Deg3Class}. 

    Let $c\in X_1(30)(\Q)$ be a fixed cusp and $\tilde c$ its reduction modulo $7$.
    Let 
    $$f:X_1(30)^{(4)}(\F_7) \rightarrow J_1(30)(\F_7) \quad \text{ be defined by }\quad  D\mapsto [D-4\tilde c].$$
    For any $z\in X_1(30)^{(4)}(\F_p)$ that is a reduction of a $x\in X_1(30)^{(4)}(\Q)$, the divisor class $f(z)$ has to be contained in the image of reduction modulo $7$ of $J_1(30)(\Q)$. 

    We compute (see \github{X1_30.m}{Claims 1-6}) that the number of such $z\in X_1(30)^{(4)}(\F_p)$ is equal to the number of cuspidal divisors $C\in X_1(30)^{(4)}(\Q)$, from which it follows that all the $z\in X_1(30)^{(4)}(\F_p)$ that are reductions of rational divisors are cuspidal.

\end{proof}

\subsection{The group $\tg{3}{15}$}

We show that $\tg{3}{15}$ is not a quartic torsion group by considering it over finite fields of different characteristics. 

\begin{proposition}\label{prop:3x15}
    The group $\tg{3}{15}$ is not a quartic torsion group.
\end{proposition}
\begin{proof}
    Let $X:=X_1(3,15)$; we consider this curve over $K:=\Q(\sqrt {-3})$. The curve $X$ has 32 cusps in total, all of which are defined either over $K$ or over $\Q(\zeta_{15})$.
    
     Let $C(K)$ denote the group generated by differences of $K$-rational cusps, and $C(\F_4)$ the subgroup generated by differences of $\F_4$-rational cusps. We have $\#C(\F_4)\leq \#C(K)$, since all $\F_4$-rational cusps are reductions of $K$-rational cusps. We compute $J(\F_7)[2]\simeq (\torz{2})^3$ \github{X1_3_15.m}{Claim 1} and $C(\F_4)[2]\simeq (\Z/2\Z)^3$ \github{X1_3_15.m}{Claim 2}. Since 
    $$(\torz{2})^3\simeq C(\F_4)[2] \twoheadleftarrow C(K)[2] \hookrightarrow J(K)[2] \hookrightarrow J(\F_7)[2] \simeq (\torz{2})^3, $$
    we conclude that the reduction modulo the prime above $2$ is injective.  

    Now the case follows \Cref{prop_main} b.ii) and \Cref{lem:main} using $K=\Q(\zeta_3)$ and $d=2$ and the fact that all the points in $X(\F_4)$ and $X(\F_{16})$ are cusps.   
\end{proof}

\begin{remark} Let $C$ be the cuspidal subgroup of $J:=J_1(m,n)$. Note that $C(\Qbar)\leq J(\Qbar)_{\tor}$, by Manin-Drinfeld, and that any prime of good reduction of $J$ is unramified in $\Q(\zeta_n)$, over which all the cusps of $X_1(m,n)$ are defined. This implies that reduction $C(\Qbar)\rightarrow C(\overline \F_p)$ is injective for all primes $p>2$ not dividing $n$.

Let $C$ be the cuspidal subgroup of $J_1(3,15)$. We compute that 
$$C(\Q(\zeta_{15}))[2]\simeq(\torz{2})^6, \text{ while }C(\F_{16})[2]\simeq(\torz{2})^5,$$ which implies that the reduction modulo a prime above $2$ is not injective on $C(\Qbar)$. This is the first instance that we are aware of, where the reduction of the cuspidal subgroup modulo a prime of good reduction is not injective on a modular Jacobian. 
\end{remark}

\section{The positive rank cases}

\subsection{The global method}

We can apply \Cref{prop:hecke_gon} and \Cref{cor:global} to eliminate most of the positive rank cases. 

\begin{proposition} \label{prop:global_argument}
    $\torz{n}$ is not a quartic torsion group for 
    $$n\in \{77, 85,91,121, 143,169, 187, 221, 289\}.$$
\end{proposition}
\begin{proof}
    The idea is to apply \Cref{cor:global}. For $n\in\{77, 121, 169, 221, 289\}$ we can check that $$\rk J_1(n)(\Q)=\rk J_0(n)(\Q)$$
    by verifying\footnote{For example, for $n=77$, one can go to \url{ https://www.lmfdb.org/ModularForm/GL2/Q/holomorphic/?level_type=divides&level=77&weight=2&char_order=2-&showcol=analytic_rank}} that all the newforms of weight 2 and level dividing $n$ with non-trivial character have analytic rank 0, and applying \Cref{thm:Kato}.
    It follows that $A:=(\diamondop{a}-1)J_1(n)(\Q)$ is finite for any nontrivial element $a\in (\torz{n})^\times$. For $n\in\{85, 91, 143,187\}$ we have $\rk J_1(n)(\Q)>\rk J_0(n)(\Q)$ and we need to be more careful with the choice of $a$.

    In these cases, the positive rank parts of $J_1(n)/J_0(n)$ come from the Galois orbits of the newforms \lmfdbnewform{85}{2}{j}{a}, \lmfdbnewform{91}{2}{g}{a}, \lmfdbnewform{91}{2}{u}{a},  \lmfdbnewform{143}{2}{h}{a} and \lmfdbnewform{187}{2}{g}{a}. Each of these newforms has an associated Nebentypus character; we have listed the LMFDB labels of the Galois orbits of these characters in \Cref{table:DS}. If we choose $a$ in the kernel of all the characters $\chi$ for which there is a positive rank newform, then $A:=(\diamondop{a}-1)J_1(n)(\Q)$ is finite, and we can apply \Cref{cor:global} as well. For $n=85$ and $91$ one can see from the LMFDB data that $a=2$, resp. $a=3$ is in the kernel of the listed characters $\chi$. For $n=143$ and $n=187$ the character has order $5$, so we choose $a$ with $\diamondop{a}$ of order co-prime to $5$ in order to ensure $\chi(a)=1$.

    We will use $q=3$ in all cases, and will have $k_3=7$ in all cases, apart from $n=85, 121, 143$ and $187$, where $k_3=8$. In \Cref{table:DS} we list the values $a \in (\torz{n})^\times $ that we use, the order of $\diamondop{a}$ and the value of  $b(n):=\frac {325} {2^{15}}\frac{[\SL_2(\mathbb Z) :\Gamma_1(n)]}{k_3 \cdot \#H}$; we need to have $b(n)> 4$ to proceed. The value $d_{CM}X_1(n) $ is the least degree of a CM point on $X_1(n)$, as listed in the GitHub repository \url{https://github.com/fsaia/least-cm-degree} of the paper \cite{CGPS22}.

\begin{table}[H]
\centering
\begin{tabular}{c|c|c|c|c|c}
$n$ & $\chi$ & $a$ & $|\diamondop{a}|$ & $\lceil b(n)\rceil $ &  $d_{CM}X_1(n)$ \\ 
\hline
$77$  & $\emptyset$            & $3$   & $30$  & $5$  & $60$  \\
$85$  & \lmfdcharorbit{85}{j}  & $2$   & $8$  & $5$  & $32$  \\
$91$  & \lmfdcharorbit{91}{g},
        \lmfdcharorbit{91}{u}  & $3$   & $6$   & $6$  & $24$  \\
$121$ & $\emptyset$            & $2$   & $55$  & $10$ & $110$ \\
$143$ & \lmfdcharorbit{143}{h} & $67$  & $12$  & $13$ & $120$ \\
$169$ & $\emptyset$            & $3$   & $39$  & $21$ & $52$  \\
$187$ & \lmfdcharorbit{187}{g} & $122$ & $16$  & $22$ & $160$ \\
$221$ & $\emptyset$            & $3$   & $48$  & $35$ & $96$  \\
$289$ & $\emptyset$            & $3$   & $272$ & $59$ & $136$ \\
\end{tabular}
\caption{}\label{table:DS}
\end{table}

We now apply \Cref{cor:global}.  Since $b(n) > 4$, \eqref{DS-inequality} is satisfied. Thus all rational points in $D\in X_1^{(4)}(n)(\Q)$, not supported on the cusps, are invariant under the action of $\diamondop{a}$. If all the points in $D$ have trivial stabilizer under the action of $\diamondop{a}$, then it has to be of degree at least $|\diamondop{a}|$ (the order of $\diamondop{a}$), which is not possible by \Cref{table:DS}. This implies that the support of $D$ consists of points corresponding to elliptic curves with $j=0$ or $j=1728$. But from the aforementioned results \cite{CGPS22}, it follows that $X_1(n)$ has no CM points of degree $< d_{CM}X_1(n) $, which shows that this is impossible. 
\end{proof}

\begin{remark} \label{rem:6.2}
As pointed out by the referee, in all cases except $n=85,91$ one could actually do without the results of \cite{CGPS22}. The reason is that the stabilizers under the action of $\langle a \rangle$ have either order $2$ or order $3$ if they occur. This implies that $|\langle a \rangle |/3$ is a lower bound on the degree of $D$. For $n=143$  one has $|\langle a \rangle |/3=4$; however, in this case one can reason as follows. As $n$ is divisible by $11$, and $11$ is not $1 \mod  3$ nor $1 \mod 4$, one cannot have a stabilizer of order $3$ nor of order $2$ and one gets a lower bound of $|\langle a \rangle |=12$ on the degree of $D$ even without using $d_{CM}X_1(n)$ to rule out CM points.
Our main reason for including the column $d_{CM}X_1(n)$ is that it improves the results for higher degree number fields that can be obtained from the proof of \Cref{prop:global_argument}. What follows from our proof is that for $n$ in the table there are no elliptic curves with a point of order $n$ over number fields of degree $d< b(n)$ such that $|\diamondop{a}|$ does not divide $d$, for the values of $n$ in \Cref{table:DS}.
\end{remark}

\subsection{The groups $\torz{63}$ and $\torz{65}$} \label{sec:63_65}

We now deal with the cases $\torz{63}$ and $\torz{65}$ which turn out to be harder and require some ad hoc argumentation. %, the main reason being that $\rk J_1(n)(\Q)>0$ for $n=63,65$. 

\begin{lemma} \label{lem:cusps}Below we determine the cusps of $X_1(n),$ for $n=63, 65$, that are defined over $\F_{2^k}$  and $\F_{3^k}$ for $k\leq 4$.
    \begin{enumerate}
        \item $n=63$: The cusps corresponding to $63$-gons are defined over $\F_2$, the cusps corresponding to $21$-gons are defined over $\F_4$, the cusps corresponding to $9$-gons are defined over $\F_8$, while the remaining cusps are defined over $\F_{2^k}$ for some $k\geq 5.$
        \item $n=65$ in characteristic $2$:  The cusps corresponding to $65$-gons are defined over $\F_2$, and the cusps corresponding to $13$-gons are defined over $\F_{16}$, while the remaining cusps are defined over $\F_{2^k}$ for some $k\geq 5.$
        \item $n=65$ in characteristic $3$:  The cusps corresponding to $65$-gons are defined over $\F_3$, and the cusps corresponding to $13$-gons are defined over $\F_{81}$, the cusps corresponding to $5$-gons are defined over $\F_{27}$ while the $1$-gons are defined over $\F_{3^k}$ for some $k\geq 5.$
    \end{enumerate}
\end{lemma}
\begin{proof}
 For any $n$, cusps corresponding to $n$-gons are defined over $\Z[1/n]$ and the $1$-gons over $\Z[1/n, \zeta_{n}+\zeta_{n}^{-1}]$. For $n=pq$, the $p$-gons are defined over $\Z[1/n, \zeta_q]$, from which the result easily follows for $n=65$.

 For $n=p^2q$ and for $d\in \{p,p^2,q,pq\}$, the $d$-gons will be defined over $\Q(\zeta_{\frac n d})$. The result for $n=63$ now follows.

\end{proof}

\begin{proposition}\label{prop:63}
    $\torz{63}$ is not a quartic torsion group.
\end{proposition}

\begin{proof}
    We will eliminate this case by showing that there are no non-cuspidal quartic point on the genus $49$ curve $Y:=X_1(63)/\diamondop{8}$ which are images of quartic points of $X_1(63)$, by following the strategy of \Cref{prop_main}. The group of diamond operators on $Y$ is $(\torz{63})^\times/ \diamondop{\pm 8}\simeq (\torz{3})^2$. This implies that isogeny factors of $J(Y)$ correspond to weight 2 newforms of level dividing 63 and with character order $1$ or $3$. The \href{https://www.lmfdb.org/ModularForm/GL2/Q/holomorphic/?level_type=divides&level=63&weight=2&char_order=1%2C3&showcol=analytic_rank.char_order.prim&hidecol=analytic_conductor.field.cm.traces.qexp} {LMFDB data} shows that all these newforms have analytic rank 0 and hence $J(Y)(\Q)$ has algebraic rank $0$.

    Suppose $E/K$ is an elliptic curve over a quartic field $K$ with a point of order $63$ over $K$ and let $x_0=(E,P)\in X_1(63)(K)$ for some $P\in E(K)$ of order 63. Let $\wp$ be a prime of $K$ over $2$, and let $\wp'$ be a prime of $\overline K$ above $\wp$. Since there are no elliptic curves with a point of order $63$ over $\F_{2^k}$ for $k\leq 4$, it follows that $x_0$ reduces to a cusp $\widetilde c\in X_1(63)(O_K/\wp)$. There is a unique cusp $c_0\in X_1(63)(\overline K)$ that reduces modulo $\wp'$ to $\widetilde c$. %Since $\widetilde{x}$ is defined over $\F_{2^k}$ for some $k\leq 4$, it follows that $c_0$ cannot correspond to a $1$-gon, since their reductions are defined over $\F_{2^6}$.

    Let $f:X_1(63)\rightarrow Y$ be the quotient map. Let $c=f(c_0)$; this is a cusp as $f$ sends cusps to cusps. Let $x=f(x_0)\in Y(K)$. The points $x$ and $c$ reduce modulo $\wp'$ to the same point in $\overline{\F}_2$; moreover, they are both defined over some $\F_{2^k}$ for $k\leq 4$. %The cusps of $Y$ (excluding $1$-gons that we don not need to consider) are defined over $\Q$ ($63$-gons and $21$-gons), $\Q(\zeta_9)^+$ (7-gons), $\Q(\zeta_7)^+$ (9-gons). The prime $2$ is inert in all these fields.

    Let $D=\sum_{i=1}^4x^{\sigma_i}$ be the degree 4 divisor which is the sum of all the conjugates of $x$; it is a $\Q$-rational divisor. As in the proof of \Cref{prop_main} there is a $\Q(\zeta_{63})$-rational cuspidal divisor $C$ such that it reduces to the same divisor as $D$ modulo a prime $\wp''$ of $\Q(\zeta_{63})$ above 2.  

    Since $2$ is unramified , it follows by \cite[Prop 2.3]{Parent00} that the kernel of the reduction mod $\wp''$ of $J(Y)(\Q)$ is either trivial or a group of exponent $2$. Hence it follows that since $[\widetilde D - \widetilde{C}]=0$, the divisor class $[D-C]\in J(Y)(\Q(\zeta_{63}))$ has to be contained in $J(Y)(\Q(\zeta_{63}))[2]$. Hence $2[D-C]=0$, so there exists a degree $8$ morphism $Y\rightarrow \PP^1$ defined over $\Q(\zeta_{63})$. But \Cref{abramovich} gives $\gon_{\Q(\zeta_{63})} Y \geq \gon_\C Y>8$, so we have arrived at a contradiction. 

\end{proof}

%This method can be applied when $\gon_\Q X> 4n$ and $\rk J_1(\Q)/J_0(\Q)=0$, where $n$ can always be taken to be $8$, and can be taken to be $6$ if the $t_2$ \todo{define this} operator kills all the $2$-torsion of $J_1$, and can be taken to be $4$ if $\rk J_1(\Q)=0$.

%\subsection{The group \texorpdfstring{$\Z/65\Z$}{Z/65Z}}

\begin{proposition}\label{prop:65}
    The group $\torz{65}$ is not a quartic torsion group.
\end{proposition}

\begin{proof}
We start similarly as in \Cref{prop_main}.  Fix a prime $\wp$ of $\Q(\zeta_{65})$ above $2$; we will denote by $\widetilde o$ the reduction of an object modulo $\wp$ or mod $2$ depending on whether the object is defined over $\Q(\zeta_{65})$ or over $\Q$.  Suppose for contradiction that $E$ is an elliptic curve over a quartic field $K$ and such that $\torz{65} \subseteq E(K) $. Let $x\in Y_1(65)(K)$ be a point representing $E$ together with a point of order $65$. Consider the corresponding points $x^{(4)}$ and $\tilde{x}^{(4)}$ on $X_1^{(4)}(65)(\Q)$ and $ X_1^{(4)}(65)(\F_2)$ respectively. Note that the Hasse-Weil bound implies that for $1\leq i \leq 4$ there are no elliptic curves over $\F_{2^i}$ with a point of order $65$. This means that $\tilde{x}^{(4)}$ has to be a sum of reductions modulo $\wp$ of $\Q(\zeta_{65})$-rational cusps. In particular, we can write $\tilde{x}^{(4)}=\widetilde C$ for some effective cuspidal divisor $C \in X_1^{(4)}(65)(\Q(\zeta_{65}))$. 
    
By looking up the analytic ranks of all weight 2 newforms of level dividing 65, it follows that $J_1(65)(\Q)$ has rank 1. In fact the \href{https://www.lmfdb.org/ModularForm/GL2/Q/holomorphic/?level_type=divides&level=65&weight=2&showcol=char_order.analytic_rank}{LMFDB data} shows more precisely that:
$$\rk(J_1(65)(\Q))=\rk(J_0(65)(\Q))= \rk(J_0^+(65)(\Q))=1.$$ 

In particular, it follows that $(\diamondop{a}-1)J_1(65)(\Q) \subseteq J_1(65)(\Q)_{tors}$ for all $a \in (\torz{65})^\times$.

    Let $c_0 \in X_1(65)(\Q)$ be the cusp whose moduli interpretation is that of a N\'eron $65$-gon and and the point of order $65$ is given by $(1,1) \in \mathbb G_m(\Q) \times \Z/65\Z$. 

    We have
    $$(\diamondop{3}-1)[x^{(4)}-C]=(\diamondop{3}-1)[x^{(4)}-4c_0]+(\diamondop{3}-1)[4c_0 -C].$$
    Since $(\diamondop{3}-1)[x^{(4)}-4c_0]\in (\diamondop{3}-1)J_1(65)(\Q) \subseteq J_1(65)(\Q)_{\tor}$ and since $(\diamondop{3}-1)(4c_0 -C)$ is a cuspidal divisor, by Manin-Drinfeld it follows that $(\diamondop{3}-1)[4c_0 -C]=[(\diamondop{3}-1)(4c_0 -C)]$ is torsion, so we conclude that $(\diamondop{3}-1)[x^{(4)}-C]\in J_1(65)(\Q(\zeta_{65}))_{\tor}$.

    By \cite[Prop 2.4]{Parent00} and the fact that $2$ is unramified in $\Q(\zeta_{65})$, the kernel of the reduction mod $2$ of $J_1(n)(\Q(\zeta_{65}))$ is either trivial or a group of exponent $2$.

     Since $(\diamondop{3}-1)[\tilde{x}^{(4)} - \widetilde{C}]=0$, the divisor class $(\diamondop{3}-1)[x^{(4)}-C]\in J_1(65)(\Q(\zeta_{65}))$ has to be contained in $J_1(65)(\Q(\zeta_{65}))[2]$. Hence $2(\diamondop{3}-1)[x^{(4)}-C]=0$; we show that this cannot happen by following the proof strategy of \Cref{prop:hecke_gon}.
     
     There are two ways one could have $2(\diamondop{3}-1)[x^{(4)}-C]=0$. Either 
     \begin{enumerate}
         \item[(i)]  $2(\diamondop{3}-1)(x^{(4)}-C)=\ddiv(f)$ for some non-constant function $f \in \C(X_1(65))$, or
         \item[(ii)] $2(\diamondop{3}-1)(x^{(4)}-C)=0$ as a divisor.
    \end{enumerate}
    
    Case (i) cannot happen. Indeed, in this case the degree of $f$ is at most $\deg(2x^{(4)}+2\diamondop{3}C) = 16$, which contradicts the gonality bound $\gon_\Q X_1(65)\geq \gon_\C X_1(65)\geq 20$ obtained using Abramovich's bound. 
    
    Now we show that the case (ii) is impossible as well. The divisor $2(\diamondop{3}-1)x^{(4)}$ is supported at non-cuspidal points, while $2(\diamondop{3}-1)C$ is supported at the cusps. The equality $2(\diamondop{3}-1)(x^{(4)}-C)=0$ forces both $2(\diamondop{3}-1)x^{(4)}=0$ and $2(\diamondop{3}-1)C=0$. In particular, one gets $\diamondop{3}x^{(4)}=x^{(4)}$. This equality can happen only if $x$ is a fixed point of $\diamondop{3^4}$. Since $3^4$ is of order $3$ in $(\torz{65})^\times/\set{\pm 1}$ this means that $E$ has to have CM by $\Q(\zeta_3)$. This is contradicting the fact that the least degree of a CM point on $X_1(65)$ is 24, as listed in the GitHub repository \url{https://github.com/fsaia/least-cm-degree} of the paper \cite{CGPS22}.
\end{proof}

\begin{remark}
    The proof of \Cref{prop:65} combines the strategies of \Cref{prop:hecke_gon,prop_main}. It should be possible to extend \Cref{prop_main} with this extra case. However, the statement of \Cref{prop_main} is already rather verbose. So for readability reasons, we chose not to extend this proposition even further with another case.
\end{remark}

\section{Higher degree number fields} \label{sec:higher_deg}

In this section we briefly explore how our methods can be applied to ruling out possible torsion groups over high degree number fields. We only use ``pure thought", using only data from LMFDB and calculations that can be done by hand. 

\begin{theorem}\label{thm:d5r1}
Let $n$ be one of the values in Table \ref{table:higher_degree}. Then there does not exist an elliptic curve $E$ with a point of order $n$ over a number field $K$ of degree $d< b(n)$ such that $|\diamondop{a}|$ does not divide $ d$ .
\begin{table}[H]
\centering
\begin{tabular}{c|c|c|c|c|c}
$n$ & $|\chi|$& $a$ & $|\diamondop{a}|$ & $\lceil b(n)\rceil $ &  $d_{CM}X_1(n) $ \\ 
\hline
$95$  &$12$        & $3^{12}$ & $3$   & $6$  & $72$  \\
$119$ &$\emptyset$ & $3$      & $48$  & $10$ & $96$  \\
$125$ &$\emptyset$ & $3$      & $50$  & $10$ & $50$  \\
$133$ &$3,6$       & $3^6$    & $3$   & $11$ & $36$  \\
$209$ &$5$         & $3^5$    & $18$  & $27$ & $180$ \\
$247$ &$12,12$     & $3^{12}$ & $3$   & $38$ & $72$  \\
$323$ &$8$         & $3^8$    & $18$  & $65$ & $288$ \\
$361$ &$\emptyset$ & $3$      & $171$ & $93$ & $114$ \\
\end{tabular}
\caption{}\label{table:higher_degree}
\end{table}
\end{theorem}
The values $n$, $a$, $|\diamondop{a}|$, $\lceil b(n)\rceil $ and $d_{CM}X_1(n)$ are the same as in \Cref{table:DS}, while the values in $|\chi|$ represent the orders of characters corresponding to the factors of $J_1(n)/J_0(n)$ of positive rank. 

\begin{proof}
    We apply \Cref{cor:global} with $H =\pm 1$, following the same strategy as in the proof of the proof of \Cref{prop:global_argument}. We choose $a:=3^k$ where $k$ is a multiple of all values in $|\chi|$ to ensure that $a$ is in the kernel of $\chi$. Using the same arguments as in \Cref{prop:global_argument}, we conclude that if $|\diamondop{a}|$ does not divide $d$, $d<d_{CM}X_1(n)$ and $d< b(n)$, then there does not exist an elliptic curve $E$ with a point of order $n$ over a number field $K$ of degree $d$.
\end{proof}

\begin{example}\label{ex:usefullness-of-H}
For $n=247$ there are exactly two Galois orbits of newforms with nontrivial Nebentypus character and positive analytic rank, namely \lmfdbnewform{247}{2}{be}{a} and  \lmfdbnewform{247}{2}{be}{b}. In both cases, the Galois orbit of the Nebentypus character is the character \lmfdcharorbit{247}{be}. Let $\chi$ denote a choice of character in this orbit. One can check that every element $a \in (\Z/247\Z)^\times$ in the kernel of $\chi$ has order $3$ in $(\Z/247\Z)^\times/{\pm 1}$. So a simple usage of \Cref{cor:global} with $H = \{\pm 1\}$ cannot rule out the existence of points on $X_1(247)$ whose degree is a multiple of $3$. There are two ways to improve on this. 

For the first improvement, notice that $(\ker \chi)/\{\pm 1 \}\subseteq (\Z/247\Z)^\times/\{\pm 1\}$ is isomorphic to $(\Z/3\Z)^2$. So we can pick two independent elements $a_1$ and $a_2$ in $(\ker \chi)/\pm 1$ and apply \Cref{cor:global} both with $a=a_1$ and $a=a_2$ to get that any point in $X_1(247)^{(d)}(\Q)$ with $d  \leq 37$ is a sum of orbits under $\ker \chi/\{\pm1 \} = \left\langle \diamondop{a_1},\diamondop{a_2} \right\rangle$. In particular, any non-cuspidal non-CM point on $X_1(247)$ of degree $\leq 37$ must have degree divisible by $9$.

For the second improvement, let $H \leq (\Z/n\Z)^\times$ be a subgroup. The points of degree $d$ on $X_1(n)$ map to points on $X_H \cong X_1(n)/H$ whose degree is a divisor of $d$. So \Cref{cor:global}, with different choices of $H$, could also be used to say something about the points on $X_1(n)$. We now continue with the case $n = 247$, and demonstrate how to rule out points of degrees 9 and 18 on $X_1(247)$ using $H=\langle -1, 77 \rangle$ and $q=3$. Observe that $77$ is an element of order 2 in $(\Z/247\Z)^\times$ and $\chi(77)=-1$. So with this choice of $H$, the isogeny factors associated to the two aforementioned newforms with positive analytic rank and character $\chi$ are not factors of $J_H$. It follows that $J_H/J_0(247)$ has rank $0$ over $\Q$ and hence $(\diamondop{a} -1)J_H(\Q)$ is finite for any element $a \in (\Z/247\Z)^\times/H$. 
Taking $H$ to be this slightly larger group has one disadvantage: the gonality bound \eqref{DS-inequality} is smaller by a factor of 2. However, we can still apply the proposition for all $d < 19$. Using the element $a=2$ of order $18$ already proves that all points of degree $<18$ have to be either cusps or CM points. In order to rule out degree 18 points as well, we proceed by applying \Cref{cor:global} again, but now with $a=68$ of order $3$, showing that in fact any effective divisor of degree $<19$ on $X_H$ consists of a sum of orbits under $\left\langle \diamondop{2},\diamondop{68} \right\rangle=(\Z/247\Z)^\times/H$, which is a group of order $54$. So any such divisor needs to consist entirely of sums of cusps and CM points.
 \end{example}

We now consider the cyclic torsion groups that need to be ruled out over quintic fields. We apply \Cref{prop_main} to eliminate some torsion groups whose corresponding modular curve has a rank 0 Jacobian over $\Q$.  

\begin{proposition}\label{prop:d5r0}
    There are no elliptic curves with a point of order $49,51,55$ or $75$ over degree $5$ number fields. 
\end{proposition}
\begin{proof}
    Let $n:=49,51,55$ or $75$. In all cases, we have $\rk J_1(n)=0.$ We apply \Cref{lem:main} and \Cref{prop_main} case b.i), using $p=2$; the necessary assumption $\gon_\Q X_1(n)>10$ follows from Abramovich's bound. In particular, we obtain $\gon_\Q X_1(n)\geq \lceil b(n) \rceil$ where $\lceil b(n) \rceil$ is as in the following table. 
    \begin{table}[H]
\centering
\begin{tabular}{c|c||c|c||c|c||c|c}
$n$ & $\lceil b(n)\rceil$& $n$ & $\lceil b(n)\rceil$&$n$ & $\lceil b(n)\rceil$&$n$ & $\lceil b(n)\rceil$ \\ 
\hline
$49$ & $12$ & $51$ & $12$ & $55$ & $15$ & $75$ & $24$ \\
\hline

\end{tabular}
\caption{}\label{table:gonalities}
\end{table}
    
\end{proof}

By \cite[Theorem 1.1.]{DerickxSutherland} the groups that appear infinitely often as quintic torsion groups are $\tg{m}{n}$ for 
\begin{equation}
    \{(1,n): 1\leq n\leq 25,\ n \neq 23 \} \cup \{(2,2n): 1\leq n\leq 8\}.
\end{equation}
To completely determine the possible torsion groups of elliptic curves over quintic fields, using the same reasoning as in \Cref{sec:overview} we would need to rule out the following 45 groups: 
\begin{itemize}

\item $\Z/n\Z, \quad n =  26, 27, 28, 30, 32, 33, 34, 35, 36, 38, 39, 40, 42, 44, 45, 48,$\\ $ 49, 50, 51,
55, 57, 63, 65, 75, 77, 85, 91, 95, 119, 121, 125, 133, 143, 169,$\\ $ 187,209, 221, 247, 289, 323, 361.$
\item $\tg{2}{2n}$, for $9\leq n\leq 12.$
\end{itemize}
 In \Cref{thm:d5r1} we eliminate $8$ groups, in \Cref{prop:d5r0} another $4$ groups, and in \Cref{rem:6.2} we have already eliminated $7$ groups. Note that the groups $\torz{4n}$ can be ruled out once the groups $\tg{2}{2n}$ are eliminated, due to the existence of a $\Q$-rational map $X_1(4n) \to X_1(2,2n)$. This allows us to ``eliminate" the groups $\torz{4n}$ for $9\leq n\leq 12.$ In total, we have eliminated 23 out of 45 possible quintic torsion groups. 
 
 Finally, note that the cases $\torz{28}$ and $\torz{30}$ \textit{do} appear as quintic torsion groups (finitely many times), see \cite{vanHoeij} and \Cref{sec:sporadic}. So these torsion groups obviously do not need to be eliminated; one instead needs to find all the degree 5 non-cuspidal points on $X_1(28)$ and $X_1(30)$.

\begin{remark} Although our methods are primarily adapted to deal with higher degree number fields, one might also wonder how well our methods would work for determining the possible torsion groups of elliptic curves in lower degree, for example over quadratic and cubic fields.

In \cite[Proof of Theorem 2.1, p.135]{KM88} Kenku and Momose list 17 cyclic groups they need to rule out to determine the possible torsion groups over quadratic fields, namely the set of groups $\torz{n}$ for    
$$n\in \{20,21,22,24,26,28,30,33,35,36,39,45,55,65,77,91,143\}.$$

We can eliminate all these groups except for $\torz{24}$ and $\torz{30}$ by a direct application of \Cref{prop_main} and \Cref{prop:hecke_sieve}, without performing any additional computations.

In \cite{Deg3Class}, to determine all the possible torsion groups of elliptic curves over cubic fields, the authors need to eliminate 15 cyclic groups $\torz{n}$ for
$$n\in \{21, 22, 24, 25, 26, 28, 30, 32, 33, 35, 36, 39, 45, 65, 121\}.$$

The groups we can eliminate by a direct application of \Cref{prop_main} and \Cref{prop:global_argument}, without performing any additional computations are $\torz{n}$ for
$$n\in \{22,25,33,35,39,45,121\}.$$

\end{remark}

\appendix

\section{Sporadic points on \texorpdfstring{$X_1(n)$}{X1(n)} and \texorpdfstring{$X_0(n)$}{X0(n)}}
\subsection{Sporadic points on \texorpdfstring{$X_1(n)$}{X1(n)}}
\label{sec:sporadic}
In the introduction we mentioned that sporadic points of degree $5\leq d\leq 13$ have been found on $X_1(n)$ in \cite{vanHoeij}. However, loc. cit. only lists points of low degree, and it is not immediately clear from the data that these points are sporadic. We give a table from which it should be clear that our claim is true, and for which $n$ there are $X_1(n)$ with degree $d$ sporadic points. The data in the table about $\rk J_1(n)$ is obtained from \cite[Theorem 3.1]{Deg3Class}, the data about $\gon_\Q X_1(n)$ from \cite{derickxVH} and the low degree points themselves are collected from \cite{vanHoeij}. Recall that for a point $x$ of degree $n$ on a curve $X$ to be sporadic, it is sufficient that either $n<\frac 1 2\gon_\Q X$ (see \cite{frey}), or $n<\gon_\Q X$ and $\rk J(X)(\Q)=0$ (see \cite[Proposition 2.3]{DerickxSutherland}).

\begin{table}[H]
\centering
\begin{tabular}{c|c|c|c}
$d$ & $n$ & $\rk J_1(n)$ & $\gon_\Q X_1(n)$  \\ 
\hline
$5$  & $28$   & $0$ & $6$    \\
$6$  & $37$   & $>0$  & $18$   \\
$7$  & $33$  & $0$  & $10$  \\
$8$ & $33$  & $0$ & $10$  \\
$9$ & $31$  & $0$ & $12$ \\
$10$ & $29$  & $0$  & $11$  \\
$11$ & $35$ & $0$ & $12$  \\
$12$ & $39$  & $0$  & $14$  \\
$13$ & $39$ & $0$  & $14$  \\

\end{tabular}
\caption{}\label{table:sporadic}
\end{table}

There are a number of points of degree $d=14,15, \ldots$, which are likely to be sporadic; however, lower bounds on the gonalities of $X_1(n)$ for $n\geq 41$ that are good enough to prove this have not been determined, and this is the obstacle of going any further than $d=13$ at the moment. 

\subsection{Sporadic points on \texorpdfstring{$X_0(n)$}{X0(n)}}
The evidence for the abundance of sporadic points on $X_0(n)$ of all degrees is even stronger than for $X_1(n)$, which leads us to make the following conjecture. 
\begin{conjecture}
For every positive integer $d$ there exists an $n$ such that there exists a sporadic point of degree $d$ on $X_0(n)$.\end{conjecture}

We now give some evidence for the conjecture.

\begin{proposition} \label{sporadic}
    If $d\leq 2166$ or there exists an imaginary quadratic number field with class number $d$, then there exists a degree $d$ sporadic point on some $X_0(n)$. 
\end{proposition}

\begin{proof}
    Let $d$ be an integer, $K=\Q(\sqrt{-n})$, where $n$ is square-free, an imaginary quadratic field whose absolute value of the discriminant is $\Delta$ and which has class number $h_\Delta=d$. Let $E_\Delta$ be an elliptic curve with complex multiplication by $\OO_K$.

 From Abramovich's bound, since $[\PSL_2(\mathbb Z) :\Gamma_0(n)]\geq n+1$, it follows that $\gon_{\Q} X_0(n)\geq \gon_{\C} X_0(n)>\frac{325n}{2^{15}}$.

By \cite[Corollary 4.2.]{Kwon} it follows that $E_\Delta$ has an $n$-isogeny over $\Q(j(E_\Delta))$, a number field of degree $[\Q(j(E_\Delta)):\Q]=h_\Delta$.

The point on $X_0(n)$ corresponding to $E_\Delta$ and this $n$-isogeny has degree exactly $h_\Delta$, as its $j$-invariant already has degree $h_\Delta$. By \cite{frey} and the bound above, this point is sporadic if
\begin{equation}\label{ineq}
d=h_\Delta<\frac{325n}{2^{16}},
\end{equation}
as then $d<\frac{1}{2}\gon_\Q X_0(n)$.

We now bound $h_\Delta$ from above. Suppose $\Delta>4$ and let $\chi_\Delta$ be the quadratic character of conductor $\Delta$ corresponding to $K$. The class number formula gives $h_\Delta=\frac{\sqrt{\Delta}}{\pi}L(1,\chi_\Delta)$. Put $S(x):=\sum_{m\leq x}\chi_\Delta(m)$. Since $S(\Delta)=0$ and $|S(x)|\leq \Delta$ for all $x$, partial summation gives
$$\left|\sum_{m>\Delta}\frac{\chi_\Delta(m)}{m}\right|=\left|\sum_{m>\Delta}\frac{S(m)}{m(m+1)}\right|\leq \frac{\Delta}{\Delta+1}<1,$$
while $\sum_{m\leq \Delta}\frac{1}{m}\leq 1+\log \Delta$. Using $\Delta\leq 4n$, it follows that
\begin{equation}\label{ineq2}
h_\Delta< \frac{\sqrt{\Delta}}{\pi}(\log \Delta+2)\leq \frac{2\sqrt{n}}{\pi}(\log (4n)+2).
\end{equation}
The function $n\mapsto (\log (4n)+2)/\sqrt{n}$ is decreasing for $n\geq 1$. For $n=2^{23}$ the right-hand side of \eqref{ineq2} is less than $35640$, while $\frac{325\cdot 2^{23}}{2^{16}}=41600$. Hence \eqref{ineq} holds for all $n\geq 2^{23}$.

For the remaining $n$ we compute the class number $h_\Delta$ of $\Q(\sqrt{-n})$ for all $5\,099\,650$ square-free $n\leq 2^{23}$. The largest class number for which \eqref{ineq} fails is $832$, attained for $n=162686$. It follows that \eqref{ineq} holds for every imaginary quadratic field with class number $d>832$, which proves the result if $d>832$ is a class number. The same computation gives, for every $2\leq d\leq 2166$, a square-free $n\leq 2192831$ such that $\Q(\sqrt{-n})$ has class number $d$ and \eqref{ineq} is satisfied.

It remains to deal with $d=1$. Let $E$ be an elliptic curve with CM by the maximal order of $\Q(\sqrt{-163})$, which has class number $1$. By \cite[Corollary 4.2.]{Kwon}, $E$ has a $163$-isogeny over $\Q$, giving a rational non-cuspidal point on $X_0(163)$. The curve $X_0(163)$ has genus $13$, so by Faltings' theorem \cite{faltings83} it has only finitely many rational points and this point is sporadic.
\end{proof}

It is widely believed that the assumptions of \Cref{sporadic} are satisfied, i.e., that every positive integer is a class number. In particular, if one defines $\mathcal{F}(d)$ to be the number of imaginary quadratic fields with class number $d$, then it is expected (see \cite[(1.4)]{Soundarajan}, and also \cite {HJKMP2019}) that
$$\mathcal{F}(d)\asymp \frac{d}{\log d}.$$ Finally, we note that if one assumes the Generalized Riemann Hypothesis, then the assumptions of \Cref{sporadic} are satisfied for all $d\leq 10^6$ \cite[Section 9]{HJKMP2019}.

\section{A moduli problem with 2 torsion}
For our explicit computations in \Cref{cor:2_2n}, we need algebraic models for the modular curves $X_1(2,2n)$. In \cite[\S 3.1]{DerickxSutherland17} there is a description of how to compute a model of $X_1(2,2n)$. The idea is to start with a suitable family $E$ of elliptic curves over $\A^2$ possessing a point $P$ of order $2$ and a point $Q$ of infinite order and writing down equations for the closed subscheme $Y \subseteq\A^2$ over which $Q$ has order exactly $2n$. The elliptic curve $E_Y$ over $Y$ will then have points $P_Y$ and $Q_Y$ of order $2$ and $2n$ respectively. If $P_Y$ and $nQ_Y$ generate $E_Y[2]$ in all fibres, then by the universal property of $Y_1(2,2n)$ one gets a map $Y \to Y_1(2,2n)$. If this map $Y \to Y_1(2,2n)$ is birational, one can construct $X_1(2,2n)$ as the desingularisation of a projective closure of $Y$. In \cite[\S 3.1]{DerickxSutherland17} the family $E$ of elliptic curves used is the one from \cite{Jain10}. However, neither from the arguments in \cite{Jain10} nor those in \cite[\S 3.1]{DerickxSutherland17} is it a priori clear whether $Y \to Y_1(2,2n)$ is birational. In order to remedy this, we write down a universal family of elliptic curves with points $P$ of order $2$ and $Q$ of order $>2$ (possibly infinite) over some open subset $Y_1(2, {>}2) \subset \A^2$ in \Cref{thm:univ_x_1_2_2n}. Carrying out the same construction with this universal family solves the issue, since the universal property now also guarantees the existence of a map $Y_1(2,2n) \to Y_1(2, {>}2)$ whose image lands in $Y$. Luckily, it turns out that this universal family is birational to the family from \cite{Jain10} as discussed in \Cref{rem:birational_families}, implying that the equations obtained \cite[\S 3.1]{DerickxSutherland17} are correct.

Additionally, the proof \Cref{thm:univ_x_1_2_2n} is also used behind the scenes in \Cref{cor:2_2n} to compute the Hecke operators directly by using their definition in terms of the moduli interpretation, which is not clear to us how to do using the family in \cite{Jain10}.

Let $R := \Spec \Z[1/2]$ and consider the category $\Ell_R$ whose objects are pairs $(E,T)$ where $T$ is an $R$ scheme and $E$ is an elliptic curve over $T$; see \cite[Section 4.13]{KatzMazur85}. A morphism in $\Ell_R$ between $f_1: E_1 \to T_1$ and $f_2: E_2 \to T_2$ is a Cartesian square
	\[
	\begin{tikzcd}
	E_1 \arrow[r,"h"] \arrow[d,"f_1"] & E_2 \arrow[d, "f_2"] \\
	T_1 \arrow[r, "g"'] & T_2.
	\end{tikzcd}
	\] A moduli problem on $\Ell_R$ is a contravariant functor $\mathcal P : \Ell_R^{op} \to Sets$. We define $[\Gamma_1(2,{>}2)] : \Ell/R \to Sets$ to be the moduli problem which sends $(E,T)$ to the set of pairs $P,Q \in E(T)$ such that $P$ is of exact order $2$ in all fibers and $Q$ is of order $>2$ in all fibers.

Define 
\begin{align*}
\Delta(b,c) &:= 16 b^2 (b + c - 1)^2(4b^2 + 4bc - 4b + c^2) \in \Z[b,c], \\
Y_1(2,{>}2) &:= \Spec \Z\left[\frac 1 2, b,c, \frac 1 {\Delta(b,c)}\right], \\
E_1(2,{>}2) &: y^2 =x^3 + c x^2 + (1-b-c)bx,  \\
P_1(2,{>}2) &:= (0,0),\\
Q_1(2,{>}2) &:= (b,b).
\end{align*}

\begin{theorem}\label{thm:univ_x_1_2_2n}
The curve $E_1(2,{>}2)$ is an elliptic curve over $Y_1(2,{>}2)$. The points $P_1(2,{>}2)$ and $Q_1(2,{>}2)$ on $E_1(2,{>}2)$ have order $2$ respectively $>2$ in all fibers.
Furthermore, the triple $$(E_1(2,{>}2),P_1(2,{>}2),Q_1(2,{>}2))$$ over $Y_1(2,{>}2)$ is universal in the sense that it represents the moduli problem $[\Gamma_1(2,{>}2)]$.
\end{theorem}

\begin{proof}
We start with $E/S$ written in Weierstrass form
\begin{equation}
    y^2+a_1xy+a_3y=x^3+a_2x^2+a_4x+a_6,
\end{equation}
together with two points $P,Q\in E(S)$ of order $2$ in all fibers and order $>2$ in all fibers, respectively. This equation can be mapped via an isomorphism $f$ to another equation in 
Weierstrass form, where the isomorphism $f$ is represented by a quadruple $(u,r,s,t)$, such that $f$ maps $(x,y)$ to 
$$((x-r)u^2, \; (y - s(x-r) - t)u^3),$$
see \cite[Proposition VII.1.3]{silverman}. The effect of this isomorphism on the coefficients $a_i$ is
\begin{equation} \label{eq:trans}
    \begin{split}a_1' &= (a_1 + 2s) u, \\
a_2' &= (a_2 - a_1s + 3r - s^2)u^2, \\
a_3' &= (a_3 + a_1r + 2t)u^3, \\
a_4' &= (a_4 + 2a_2r - a_1(rs+t) - a_3s + 3r^2 - 2st)u^4, \\
a_6' &= (a_6 - a_1rt + a_2r^2 - a_3t + a_4r + r^3 - t^2)u^6 .\end{split}
\end{equation}

We move the point $P$ to $(0,0)$ by selecting an appropriate $r$ (to move $x(P)$ to $0$) and $t$ (to move $y(P)$ to $0$, after $x$ has been moved to $0$). Now our equation is of the form
$$y^2+a_1'xy=x^2+a_2'x^2+a_4'x.$$ Notice that $a_6'=0$ since $(0,0)$ is a point on the curve and $a_3'=0$ since $(0,0)$ is of order 2.  We can move $a_1'$ to $0$ by selecting an appropriate $s$ in the quadruple corresponding to the isomorphism by \eqref{eq:trans}. Now the point $Q$ is of the form $(v,w)$ for some $v,w \in \OO_S(S)$ since it is of order $>2$ in all fibers. We can scale the coordinates of $Q$ by $(v,w)\mapsto(vu^2, wu^3)$. Hence, we can choose $u=w/v$ and letting $b:=v^3/w^2$, we get that $Q$ is now of the form $(b,b)$. 
Note here that $w,v,u\in \OO_S(S)^\times$, since otherwise $Q$ would be a point of order $2$ in the fiber where one of $w,v,u$ vanishes. 
Let 
$$y^2=x^3+a_2''x^2+a_4''x$$
be the equation we get after this transformation.
Let $c=a_2''$. Since $Q=(b,b)$ lies on this equation we get that
$$b^2=b^3+cb^2+ba_4'',$$
so $a_4''=\frac{b^2-b^3-cb^2}{b}=(1-b-c)b$, as claimed. 
\end{proof}

\begin{remark}\label{rem:birational_families}
Jain gives in \cite{Jain10} a (two-variable) family of elliptic curves that is birational to $Y_1(2, {>}2)$. His family is given in terms of variables $t$ and $q$, and to get from $Y_1(2, {>}2)$ (given in terms of $b$ and $c$) to his family and vice versa, one uses the following transformations:
\begin{align*}
t := (-2b - c + 1)/(c-1);\\
q := b+1/2c+1/2-2b/(2b + c - 1);\\
b := (t+1)(qt+1)/t^2;\\
c := (t^2-2qt-2)/t^2.
\end{align*}
\end{remark}

\listoftodos[List of todo's]{}
\todo[inline]{remove todolist}

\bibliographystyle{siam}
\bibliography{bibliography1}

@Article{waterhouse69,
 Author = {Waterhouse, W. C.},
 Title = {Abelian varieties over finite fields},
 FJournal = {Annales Scientifiques de l'{\'E}cole Normale Sup{\'e}rieure. Quatri{\`e}me S{\'e}rie},
 Journal = {Ann. Sci. {\'E}c. Norm. Sup{\'e}r. (4)},
 ISSN = {0012-9593},
 Volume = {2},
 Pages = {521--560},
 Year = {1969},
 Language = {English},
 DOI = {10.24033/asens.1183},
 zbMATH = {3301167},
 Zbl = {0188.53001}
}

@unpublished {vanHoeij,
    AUTHOR = {Mark van Hoeij},
     TITLE = {Low degree places on the modular curve ${X}_1(N)$},
   NOTE = {preprint, available at \url{https://arxiv.org/abs/1202.4355}}
}

@book {silverman-advanced,
    AUTHOR = {Silverman, Joseph H.},
     TITLE = {Advanced topics in the arithmetic of elliptic curves},
    SERIES = {Graduate Texts in Mathematics},
    VOLUME = {151},
 PUBLISHER = {Springer-Verlag},
   ADDRESS = {New York},
      YEAR = {1994},
     PAGES = {xiv+525},
      ISBN = {0-387-94328-5},
   MRCLASS = {11G05 (11G07 11G15 11G40 14H52)},
MRREVIEWER = {Henri Darmon},
       DOI = {10.1007/978-1-4612-0851-8},
       URL = {http://dx.doi.org/10.1007/978-1-4612-0851-8},
}

@InCollection{Kato2004,
 Author = {Kato, Kazuya},
 Title = {{{\(p\)}}-adic {Hodge} theory and values of zeta functions of modular forms.},
 BookTitle = {Cohomologies \(p\)-adiques et applications arithm\'etiques (III)},
 ISBN = {2-85629-158-9},
 Pages = {117--290},
 Year = {2004},
 Publisher = {Paris: Soci{\'e}t{\'e} Math{\'e}matique de France},
 Language = {English},
 zbMATH = {2123616},
 Zbl = {1142.11336}
}

@article {mazur77,
    AUTHOR = {Mazur, B.},
     TITLE = {Modular curves and the {E}isenstein ideal},
   JOURNAL = {Inst. Hautes \'Etudes Sci. Publ. Math.},
  FJOURNAL = {Institut des Hautes \'Etudes Scientifiques. Publications
              Math\'ematiques},
    NUMBER = {47},
      YEAR = {1977},
     PAGES = {33--186 (1978)},
      ISSN = {0073-8301},
     CODEN = {PMIHA6},
MRREVIEWER = {M. Ohta},
       URL = {http://www.numdam.org/item?id=PMIHES_1977__47__33_0},
}

@article {kamienny92,
    AUTHOR = {Kamienny, S.},
     TITLE = {Torsion points on elliptic curves and {$q$}-coefficients of
              modular forms},
   JOURNAL = {Invent. Math.},
  FJOURNAL = {Inventiones Mathematicae},
    VOLUME = {109},
      YEAR = {1992},
    NUMBER = {2},
     PAGES = {221--229},
      ISSN = {0020-9910},
     CODEN = {INVMBH},
MRREVIEWER = {Glenn Stevens},
       DOI = {10.1007/BF01232025},
       URL = {http://dx.doi.org/10.1007/BF01232025},
}

@article {KM88,
    AUTHOR = {Kenku, M. A. and Momose, F.},
     TITLE = {Torsion points on elliptic curves defined over quadratic
              fields},
   JOURNAL = {Nagoya Math. J.},
  FJOURNAL = {Nagoya Mathematical Journal},
    VOLUME = {109},
      YEAR = {1988},
     PAGES = {125--149},
      ISSN = {0027-7630},
     CODEN = {NGMJA2},
MRREVIEWER = {Bert van Geemen},
       URL = {http://projecteuclid.org/euclid.nmj/1118780896},
}

@preamble{
   "\def\cprime{$'$} "
}

@book {silverman,
    AUTHOR = {Silverman, Joseph H.},
     TITLE = {The arithmetic of elliptic curves},
    SERIES = {Graduate Texts in Mathematics},
    VOLUME = {106},
   EDITION = {Second},
 PUBLISHER = {Springer, Dordrecht},
      YEAR = {2009},
     PAGES = {xx+513},
      ISBN = {978-0-387-09493-9},
   MRCLASS = {11-02 (11G05 11G20 14H52 14K15)},
MRREVIEWER = {Vasil{\cprime} {\=I}. Andr{\={\i}}{\u\i}chuk},
       DOI = {10.1007/978-0-387-09494-6},
       URL = {http://dx.doi.org/10.1007/978-0-387-09494-6},
}

@book {modular,
    AUTHOR = {Diamond, Fred and Shurman, Jerry},
     TITLE = {A first course in modular forms},
    SERIES = {Graduate Texts in Mathematics},
    VOLUME = {228},
 PUBLISHER = {Springer-Verlag, New York},
      YEAR = {2005},
     PAGES = {xvi+436},
      ISBN = {0-387-23229-X},
   MRCLASS = {11Fxx},
MRREVIEWER = {Henri Darmon},
}

@Article{DKSS,
 Author = {Derickx, Maarten and Kamienny, Sheldon and Stein, William and Stoll, Michael},
 Title = {Torsion points on elliptic curves over number fields of small degree},
 FJournal = {Algebra \& Number Theory},
 Journal = {Algebra Number Theory},
 ISSN = {1937-0652},
 Volume = {17},
 Number = {2},
 Pages = {267--308},
 Year = {2023},
 Language = {English},
 DOI = {10.2140/ant.2023.17.267},
 zbMATH = {7669194},
 Zbl = {1536.11084}
}

@article {katz81,
    AUTHOR = {Katz, Nicholas M.},
     TITLE = {Galois properties of torsion points on abelian varieties},
   JOURNAL = {Invent. Math.},
  FJOURNAL = {Inventiones Mathematicae},
    VOLUME = {62},
      YEAR = {1981},
    NUMBER = {3},
     PAGES = {481--502},
      ISSN = {0020-9910},
     CODEN = {INVMBH},
MRREVIEWER = {Kenneth A. Ribet},
       DOI = {10.1007/BF01394256},
       URL = {http://dx.doi.org/10.1007/BF01394256},
}

@article {derickxVH,
    AUTHOR = {Derickx, Maarten and van Hoeij, Mark},
     TITLE = {Gonality of the modular curve {$X_1(N)$}},
   JOURNAL = {J. Algebra},
  FJOURNAL = {Journal of Algebra},
    VOLUME = {417},
      YEAR = {2014},
     PAGES = {52--71},
      ISSN = {0021-8693},
   MRCLASS = {14G35 (14G25 14H25 14H51)},
MRREVIEWER = {Christian Frederik Wei\ss },
       DOI = {10.1016/j.jalgebra.2014.06.026},
       URL = {https://doi.org/10.1016/j.jalgebra.2014.06.026},
}

@article {najman16,
    AUTHOR = {Najman, F.},
     TITLE = {Torsion of rational elliptic curves over cubic fields and
              sporadic points on {$X_1(n)$}},
   JOURNAL = {Math. Res. Lett.},
  FJOURNAL = {Mathematical Research Letters},
    VOLUME = {23},
      YEAR = {2016},
    NUMBER = {1},
     PAGES = {245--272},
      ISSN = {1073-2780},
   MRCLASS = {14H52 (11G05 11G18 14G25)},
MRREVIEWER = {John L. Boxall},
       DOI = {10.4310/MRL.2016.v23.n1.a12},
       URL = {http://dx.doi.org/10.4310/MRL.2016.v23.n1.a12},
}

@article {DerickxSutherland,
    AUTHOR = {Derickx, Maarten and Sutherland, Andrew V.},
     TITLE = {Torsion subgroups of elliptic curves over quintic and sextic
              number fields},
   JOURNAL = {Proc. Amer. Math. Soc.},
  FJOURNAL = {Proceedings of the American Mathematical Society},
    VOLUME = {145},
      YEAR = {2017},
    NUMBER = {10},
     PAGES = {4233--4245},
      ISSN = {0002-9939},
   MRCLASS = {11G05 (11G18 14G35 14H51 14H52)},
MRREVIEWER = {Jie Shu},
       DOI = {10.1090/proc/13605},
       URL = {https://doi.org/10.1090/proc/13605},
}

@Article{sutherlandtorsion,
 Author = {Sutherland, Andrew V.},
 Title = {Constructing elliptic curves over finite fields with prescribed torsion},
 FJournal = {Mathematics of Computation},
 Journal = {Math. Comput.},
 ISSN = {0025-5718},
 Volume = {81},
 Number = {278},
 Pages = {1131--1147},
 Year = {2012},
 Language = {English},
 DOI = {10.1090/S0025-5718-2011-02538-X},
 zbMATH = {6028401},
 Zbl = {1267.11074}
}

@unpublished{derickxstoll,
 Author = {Derickx, Maarten and Stoll, Michael},
 Title = {Prime order torsion on elliptic curves over number fields. Part I: Asymptotics},
NOTE = {preprint, available at \url{https://arxiv.org/abs/2505.14109}},
}

@article {faltings83,
    AUTHOR = {Faltings, G.},
     TITLE = {Endlichkeitss\"{a}tze f\"{u}r abelsche {V}ariet\"{a}ten \"{u}ber
              {Z}ahlk\"{o}rpern},
   JOURNAL = {Invent. Math.},
  FJOURNAL = {Inventiones Mathematicae},
    VOLUME = {73},
      YEAR = {1983},
    NUMBER = {3},
     PAGES = {349--366},
      ISSN = {0020-9910},
   MRCLASS = {11D41 (11G30 14G25)},
  MRNUMBER = {718935},
MRREVIEWER = {James Milne},
       DOI = {10.1007/BF01388432},
       URL = {https://doi.org/10.1007/BF01388432},
}

@article {frey,
    AUTHOR = {Frey, Gerhard},
     TITLE = {Curves with infinitely many points of fixed degree},
   JOURNAL = {Israel J. Math.},
  FJOURNAL = {Israel Journal of Mathematics},
    VOLUME = {85},
      YEAR = {1994},
    NUMBER = {1-3},
     PAGES = {79--83},
      ISSN = {0021-2172},
   MRCLASS = {11G30 (11G05 14G25 14H25)},
MRREVIEWER = {Takeshi Ooe},
       DOI = {10.1007/BF02758637},
       URL = {https://doi.org/10.1007/BF02758637},
}

@article {abramovich,
    AUTHOR = {Abramovich, Dan},
     TITLE = {A linear lower bound on the gonality of modular curves},
   JOURNAL = {Internat. Math. Res. Notices},
  FJOURNAL = {International Mathematics Research Notices},
      YEAR = {1996},
    NUMBER = {20},
     PAGES = {1005--1011},
      ISSN = {1073-7928},
   MRCLASS = {11G18 (11F32 14G35)},
MRREVIEWER = {M. Ram Murty},
       DOI = {10.1155/S1073792896000621},
       URL = {https://doi.org/10.1155/S1073792896000621},
}

@Article{Deg3Class,
 Author = {Maarten {Derickx} and Anastassia {Etropolski} and Mark {van Hoeij} and Jackson S. {Morrow} and David {Zureick-Brown}},
 Title = {{Sporadic cubic torsion}},
 FJournal = {{Algebra \& Number Theory}},
 Journal = {{Algebra Number Theory}},
 ISSN = {1937-0652},
 Volume = {15},
 Number = {7},
 Pages = {1837--1864},
 Year = {2021},
 Publisher = {Mathematical Sciences Publishers (MSP), Berkeley, CA},

 DOI = {10.2140/ant.2021.15.1837},
 MSC2010 = {11G18 11G05 11Y50 14H45}
}

@article {BruinNajman15,
    AUTHOR = {Bruin, Peter and Najman, Filip},
     TITLE = {Hyperelliptic modular curves {$X_0(n)$} and isogenies of
              elliptic curves over quadratic fields},
   JOURNAL = {LMS J. Comput. Math.},
  FJOURNAL = {LMS Journal of Computation and Mathematics},
    VOLUME = {18},
      YEAR = {2015},
    NUMBER = {1},
     PAGES = {578--602},
   MRCLASS = {11G05 (11-04 11G18 14G25)},
  MRNUMBER = {3389884},
MRREVIEWER = {Piotr Kraso\'{n}},
       DOI = {10.1112/S1461157015000157},
       URL = {https://doi.org/10.1112/S1461157015000157},
}

@incollection {stein07,
    AUTHOR = {William Arthur Stein},
     TITLE = {Modular forms, a computational approach},
   JOURNAL = {Graduate Studies in Mathematics},
    VOLUME = {79},
     PUBLISHER = {Amer. Math. Soc., Providence, RI},
     YEAR = {2007},
}

@Article{CGPS22,
 Author = {Clark, Pete L. and Genao, Tyler and Pollack, Paul and Saia, Frederick},
 Title = {The least degree of a {CM} point on a modular curve},
 FJournal = {Journal of the London Mathematical Society. Second Series},
 Journal = {J. Lond. Math. Soc., II. Ser.},
 ISSN = {0024-6107},
 Volume = {105},
 Number = {2},
 Pages = {825--883},
 Year = {2022},
 Language = {English},
 DOI = {10.1112/jlms.12518},
 zbMATH = {7731379}
}

@Article{bruin_najman2016,
 Author = {Peter {Bruin} and Filip {Najman}},
 Title = {{A criterion to rule out torsion groups for elliptic curves over number fields}},
 FJournal = {{Research in Number Theory}},
 Journal = {{Res. Number Theory}},
 ISSN = {2522-0160},
 Volume = {2},
 Pages = {13},
 Note = {Id/No 3},
 Year = {2016},
 Publisher = {Springer International Publishing (SpringerOpen), Cham},

 DOI = {10.1007/s40993-015-0031-5},
 MSC2010 = {11G05},
 Zbl = {1408.11053}
}

@Article{JeonKimSchweizer04,
 Author = {Jeon, Daeyeol and Kim, Chang Heon and Schweizer, Andreas},
 Title = {On the torsion of elliptic curves over cubic number fields},
 FJournal = {Acta Arithmetica},
 Journal = {Acta Arith.},
 ISSN = {0065-1036},
 Volume = {113},
 Number = {3},
 Pages = {291--301},
 Year = {2004},
 Language = {English},
 DOI = {10.4064/aa113-3-6},
 URL = {journals.impan.gov.pl/aa/Inf/113-3-6.html},
 zbMATH = {2092487},
 Zbl = {1083.11038}
}

@Article{JeonKimPark06,
 Author = {Jeon, Daeyeol and Kim, Chang Heon and Park, Euisung},
 Title = {On the torsion of elliptic curves over quartic number fields},
 FJournal = {Journal of the London Mathematical Society. Second Series},
 Journal = {J. London Math. Soc. (2)},
 ISSN = {0024-6107},
 Volume = {74},
 Number = {1},
 Pages = {1--12},
 Year = {2006},
 Language = {English},
 DOI = {10.1112/S0024610706022940},
 zbMATH = {5057210},
 Zbl = {1165.11054}
}

@Article{DerickxSutherland17,
 Author = {Derickx, Maarten and Sutherland, Andrew V.},
 Title = {Torsion subgroups of elliptic curves over quintic and sextic number fields},
 FJournal = {Proceedings of the American Mathematical Society},
 Journal = {Proc. Am. Math. Soc.},
 ISSN = {0002-9939},
 Volume = {145},
 Number = {10},
 Pages = {4233--4245},
 Year = {2017},
 Language = {English},
 DOI = {10.1090/proc/13605},
 zbMATH = {6767082},
 Zbl = {1421.11049}
}

@Article{magma,
 Author = {Bosma, Wieb and Cannon, John and Playoust, Catherine},
 Title = {The {Magma} algebra system. {I}: {The} user language},
 FJournal = {Journal of Symbolic Computation},
 Journal = {J. Symb. Comput.},
 ISSN = {0747-7171},
 Volume = {24},
 Number = {3-4},
 Pages = {235--265},
 Year = {1997},
 Language = {English},
 DOI = {10.1006/jsco.1996.0125},
 URL = {semanticscholar.org/paper/b7ebcec93f35e52ef14a41c9c9d55fa132987473},
 zbMATH = {1077111},
 Zbl = {0898.68039}
}

@InCollection{DiamondIm,
 Author = {Diamond, Fred and Im, John},
 Title = {Modular forms and modular curves},
 BookTitle = {Seminar on Fermat's last theorem. The Fields Institute for Research in Mathematical Sciences, 1993-1994, Toronto, Ontario, Canada. Proceedings},
 ISBN = {0-8218-0313-1},
 Pages = {39--133},
 Year = {1995},
 Publisher = {Providence, RI: American Mathematical Society (publ. for the Canadian Mathematical Society)},
 Language = {English},
 zbMATH = {825072},
 Zbl = {0853.11032}
}

@Article{Kubert76,
 Author = {Kubert, Daniel Sion},
 Title = {Universal bounds on the torsion of elliptic curves},
 FJournal = {Proceedings of the London Mathematical Society. Third Series},
 Journal = {Proc. Lond. Math. Soc. (3)},
 ISSN = {0024-6115},
 Volume = {33},
 Pages = {193--237},
 Year = {1976},
 Language = {English},
 DOI = {10.1112/plms/s3-33.2.193},
 zbMATH = {3517292},
 Zbl = {0331.14010}
}

@Article{Kwon,
 Author = {Kwon, Soonhak},
 Title = {Degree of isogenies of elliptic curves with complex multiplication},
 FJournal = {Journal of the Korean Mathematical Society},
 Journal = {J. Korean Math. Soc.},
 ISSN = {0304-9914},
 Volume = {36},
 Number = {5},
 Pages = {945--958},
 Year = {1999},
 Language = {English},
 zbMATH = {1394559},
 Zbl = {0943.11031}
}

@Article{Parent00,
 Author = {Parent, Pierre},
 Title = {Torsion des courbes elliptiques sur les corps cubiques. ({Torsion} of elliptic curves over cubic fields)},
 FJournal = {Annales de l'Institut Fourier},
 Journal = {Ann. Inst. Fourier},
 ISSN = {0373-0956},
 Volume = {50},
 Number = {3},
 Pages = {723--749},
 Year = {2000},
 Language = {French},
 DOI = {10.5802/aif.1770},
 zbMATH = {1478801},
 Zbl = {0971.11030}
}

@Article{Kim03,
 Author = {Kim, Henry H.},
 Title = {Functoriality for the exterior square of {{\(\text{GL}_4\)}} and the symmetric fourth of {{\(\text{GL}_2\)}}. {Appendix} 1 by {Dinakar} {Ramakrishnan}, {Appendix} 2 by {Henry} {H}. {Kim} and {Peter} {Sarnak}},
 FJournal = {Journal of the American Mathematical Society},
 Journal = {J. Am. Math. Soc.},
 ISSN = {0894-0347},
 Volume = {16},
 Number = {1},
 Pages = {139--183},
 Year = {2003},
 Language = {English},
 DOI = {10.1090/S0894-0347-02-00410-1},
 zbMATH = {1832412},
 Zbl = {1018.11024}
}

@phdthesis{DerickxPhD,
  author  = {Derickx, Maarten},
  title   = {Torsion points on elliptic curves over number fields of small degree},
  school  = {Universiteit Leiden},
  year    = {2016}}

@Article{HJKMP2019,
 Author = {Holmin, S. and Jones, N. and Kurlberg, P. and McLeman, C. and Petersen, K.},
 Title = {Missing class groups and class number statistics for imaginary quadratic fields},
 FJournal = {Experimental Mathematics},
 Journal = {Exp. Math.},
 ISSN = {1058-6458},
 Volume = {28},
 Number = {2},
 Pages = {233--254},
 Year = {2019},
 Language = {English},
 DOI = {10.1080/10586458.2017.1383952},
 zbMATH = {7068038},
 Zbl = {1470.11285}
}

@Article{Soundarajan,
 Author = {Soundararajan, K.},
 Title = {The number of imaginary quadratic fields with a given class number},
 FJournal = {Hardy-Ramanujan Journal},
 Journal = {Hardy-Ramanujan J.},
 ISSN = {2804-7370},
 Volume = {30},
 Pages = {13--18},
 Year = {2007},
 Language = {English},
 zbMATH = {5290085},
 Zbl = {1151.11053}
}

@Article{Jain10,
 Author = {Jain, Sonal},
 Title = {Points of low height on elliptic surfaces with torsion},
 FJournal = {LMS Journal of Computation and Mathematics},
 Journal = {LMS J. Comput. Math.},
 ISSN = {1461-1570},
 Volume = {13},
 Pages = {370--387},
 Year = {2010},
 Language = {English},
 DOI = {10.1112/S1461157009000151},
 zbMATH = {5947727},
 Zbl = {1278.11071}
}

@Book{KatzMazur85,
 Author = {Katz, Nicholas M. and Mazur, Barry},
 Title = {Arithmetic moduli of elliptic curves},
 FSeries = {Annals of Mathematics Studies},
 Series = {Ann. Math. Stud.},
 Volume = {108},
 Year = {1985},
 Publisher = {Princeton University Press, Princeton, NJ},
 Language = {English},
 DOI = {10.1515/9781400881710},
 zbMATH = {3920653},
 Zbl = {0576.14026}
}

@misc{HamakiotesBau,
 author = {Hamakiotes, Asimina S. and Lau, Jun Bo},
 title = {Genus formulas for families of modular curves},
 year = {2025},
 howpublished = {Preprint, {arXiv}:2501.10883 [math.{NT}] (2025)},
 url = {https://arxiv.org/abs/2501.10883},
 arXiv = {arXiv:2501.10883}
}
\end{document}